\documentclass [12pt]{amsart}
\usepackage[utf8]{inputenc}
\pdfoutput=1
\usepackage{amsmath,amssymb,amsthm,amsfonts}
\usepackage{mathtools}%

\usepackage[english]{babel}%
\usepackage{comment}%
\usepackage{hyperref}%
\hypersetup{bookmarksdepth=2}
\usepackage{bm}%
\usepackage{bbm}%
\usepackage{mathrsfs}%

\usepackage{tikz,graphicx,color}
\usepackage{tikz-cd}%
\usepackage{tikz-3dplot}%
\usetikzlibrary{calc}%
\usetikzlibrary{arrows}%
\usetikzlibrary{shapes}%
\usetikzlibrary{patterns}%
\usetikzlibrary{positioning}%
\usetikzlibrary{arrows.meta}
\usetikzlibrary{decorations.markings}
\usetikzlibrary{knots}

\usepackage{epstopdf}%

\usepackage[arrow]{xy}%
\usepackage{diagbox}%
\usepackage{subfig}%
\usepackage{arcs}%
\usepackage{xcolor}%
\usepackage{xspace}

\usepackage[margin=1.0in]{geometry}%

\usepackage{enumitem}
\usepackage{letltxmacro}
\usepackage{thmtools,etoolbox}

\def\myarabic#1{\normalfont(\roman{#1})}
\newlist{theoremlist}{enumerate}{1}
\setlist[theoremlist]{label=\myarabic{theoremlisti},ref={\myarabic{theoremlisti}},itemindent=0pt,labelindent=0pt,
  leftmargin=*,noitemsep}

\makeatletter
\renewcommand{\p@theoremlisti}{\perh@ps{\thetheorem}}
\protected\def\perh@ps#1#2{\textup{#1#2}}
\newcommand{\itemrefperh@ps}[2]{\textup{#2}}
\newcommand{\itemref}[1]{\begingroup\let\perh@ps\itemrefperh@ps\ref{#1}\endgroup}
\makeatother

\usepackage{nameref,hyperref}
\usepackage[capitalize]{cleveref}

\newtheorem{theorem}{Theorem}[section]

\newtheorem{lemma}[theorem]{Lemma}
\newtheorem{proposition}[theorem]{Proposition}
\newtheorem{corollary}[theorem]{Corollary}
\theoremstyle{definition}

\theoremstyle{definition}
\newtheorem{remark}[theorem]{Remark}
\theoremstyle{definition}
\newtheorem{definition}[theorem]{Definition}

\newtheorem{question}[theorem]{Question}
\theoremstyle{definition}

\theoremstyle{definition}
\newtheorem{example}[theorem]{Example}

\crefname{figure}{Figure}{Figures}

\def\figref#1(#2){Figure~\hyperref[#1]{\ref*{#1}(#2)}}

\addtotheorempostheadhook[theorem]{\crefalias{theoremlisti}{theorem}}
\addtotheorempostheadhook[lemma]{\crefalias{theoremlisti}{lemma}}
\addtotheorempostheadhook[proposition]{\crefalias{theoremlisti}{proposition}}
\addtotheorempostheadhook[corollary]{\crefalias{theoremlisti}{corollary}}

\def\Fcal{\mathcal{F}}\def\Mcal{\mathcal{M}}\def\Scal{\mathcal{S}}\def\Xcal{\mathcal{X}}

\def\C{{\mathbb{C}}}
\def\R{{\mathbb{R}}}

\def\Z{{\mathbb{Z}}}

\newcommand\parr[1]{{({#1})}}
\def\<{{\langle}}
\def\>{{\rangle}}
\def\eps{{\epsilon}}
\def\la{{\lambda}}

\def\Conv{ \operatorname{Conv}}

\def\proj{ \operatorname{proj}}

\def\xrasim{\xrightarrow{\sim}}

\def\Povtp_#1{\Pi_{#1}^{>0}}
\def\Povtnn_#1{\Pi_{#1}^{\geq0}}

\numberwithin{equation}{section}

\def\Space{\Xcal}
\def\Closure{\overline{\Space}}

\def\bx{{\bm x}}
\def\by{{\bm y}}
\def\bz{{\bm z}}

\crefname{figure}{Figure}{Figures}

\def\eps{\varepsilon}

\def\ggp_#1{\gamma^{\C}_{f',#1}}

\def\Ptp_#1{\Pi^{>0}_{#1}}
\def\Ptnn_#1{\Pi^{\geq0}_{#1}}

\def\t{u}

\def\Measop{\operatorname{Meas}}

\def\Meas(#1,#2){\Measop_{#1}(#2)}
\def\Measp(#1,#2){\Measop'_{#1}(#2)}

\def\Crit{\operatorname{Crit}}
\def\Ctp{\Crit^{>0}}

\makeatletter
\newcommand{\raisemath}[1]{\mathpalette{\raisem@th{#1}}}
\newcommand{\raisem@th}[3]{\raisebox{#1}{$#2#3$}}
\makeatother

\def\eps{\epsilon}

\def\Ctnn{\Crit^{\geq0}}

\def\Measd(#1,#2){\mathrm{M\widehat{e\vphantom{i}a}s}_{#1}(#2)}

\def\xrasim{\xrightarrow{\sim}}

\def\bt{{\bm t}}
\def\t{t}

\def\CCrit{\Crit}
\def\Cio{\CCrit^{\circ}}

\def\v{v}

\def\CioR_#1{\Cio_{#1}(\R)}

\def\conn_#1{c_{#1}}

\def\THtp{\Theta^{>0}}

\def\Hyp_#1{\Delta_{#1}}

\def\Cl(#1){#1^\boxtimes}

\let\ge\geqslant
\let\geq\geqslant

\let\leq\leqslant

\def\Meascl(#1,#2){{\overline{\Measop}}_{#1}(#2)}

\def\dsh#1{#1^\downarrow}
\def\Rtp{\R_{>0}}
\def\Rtnn{\R_{\ge0}}
\def\Ztnn{\Z_{\ge0}}

\def\Pmid_#1{\Pi^{>0}_{#1,\dsh{#1}}}

\def\paragraph#1{\subsubsection*{#1}}

\def\Mscl{1}

\def\crat(#1,#2;#3,#4){(\v_{#1},\v_{#2};\v_{#3},\v_{#4})}

\makeatletter
\@namedef{subjclassname@2020}{\textup{2020} Mathematics Subject Classification}
\makeatother

\def\GH_#1{G^{\Mcal}_{#1}}

\def\Ordop{{\mathscr{O}}}
\def\COrdop{{\mathscr{L}}}
\def\Ord{\Ordop}
\def\COrd{\COrdop}

\def\Orda{\Ordop}

\def\Sim{\operatorname{Sim}_1}

\def\Ass{{\mathscr{A}}}
\def\Cyc{{\mathscr{C}}}

\def\Oao{\Ord^\circ}

\def\tube{\tau}
\def\Tubing{\mathbf{T}}
\def\KAss(#1){\mathcal{K}_{\Ass}(#1)}
\def\KCyc(#1){\mathcal{K}_{\Cyc}(#1)}
\def\FCyc(#1){\mathcal{F}_{\Cyc}(#1)}

\def\Adm(#1;#2){\operatorname{Adm}(#2)}

\def\Melt{\Mcal}

\def\Pa{\tilde P}

\def\xt{\tilde x}
\def\bxt{\tilde \bx}

\def\FM_#1{\Fcal_{\Ass_\Melt}(P;#1)}
\def\FMp_#1{\Fcal_{\Ass_{\Melt'}}(P;#1)}

\def\Triples(#1){#1^{\<3\>}}

\def\PTubing{\Tubing\sqcup\{P\}}
\def\PT{\PTubing}
\def\PTp{\T'\sqcup\{P\}}
\def\HAT{\widehat\Tubing}

\def\al_#1{\alpha_{#1}}
\def\alp_#1{\alpha'_{#1}}

\def\PTTx_#1{P_{#1}}
\def\Oo{\Ordop^\circ}
\def\COo{\COrdop^\circ}
\def\D_#1{\Comp_{#1}(P)}
\def\Comp{\operatorname{Comp}}

\def\CHx(#1,#2){\widehat#2[#1]}
\def\CH(#1){\CHx(#1,\Tubing)}
\def\Px[#1]{#1}
\def\Pax[#1]{#1}

\def\Star{\operatorname{Star}}

\def\Afford(#1){\R^{|#1|-2}}

\def\Closure(#1){\overline{#1}}
\def\bnull{\mathbf{0}}

\def\coord(#1,#2){z_{#2}\tbr[#1]}
\def\coordx(#1,#2){x_{#2}\tbr[#1]}
\def\coordxp(#1,#2){x'_{#2}\tbr[#1]}
\def\coordxt(#1,#2){\tilde x_{#2}\tbr[#1]}
\def\cxi_#1(#2,#3){x^\parr{#1}_{#3}\tbr[#2]}
\def\cyi_#1(#2,#3){y^\parr{#1}_{#3}\tbr[#2]}
\def\czi_#1(#2,#3){z^\parr{#1}_{#3}\tbr[#2]}
\def\coordy(#1,#2){y_{#2}\tbr[#1]}
\def\coordyt(#1,#2){\tilde y_{#2}\tbr[#1]}
\def\coordypn(#1,#2){y^\parr n_{#2}\tbr[#1]}
\def\coordxpn(#1,#2){x^\parr n_{#2}\tbr[#1]}
\def\coordz(#1,#2){z_{#2}\tbr[#1]}
\def\coordzp(#1,#2){z'_{#2}\tbr[#1]}

\def\ATTx(#1){A(#1)}
\def\BTTx(#1){B(#1)}

\def\tmax{t^{\operatorname{max}}}

\def\DT{\D_{\Tubing}}

\def\OFace(#1,#2){\Fcal^\circ_{\Ord}(#1,#2)}
\def\COFace(#1,#2){\Fcal^\circ_{\COrd}(#1,#2)}
\def\OFacecl(#1,#2){\Fcal_{\Ord}(#1,#2)}
\def\COFacecl(#1,#2){\Fcal_{\COrd}(#1,#2)}

\def\OaFace(#1,#2){\Fcal^\circ_{\Ord}(#1,#2)}
\def\OaFacecl(#1,#2){\Fcal_{\Ord}(#1,#2)}

\def\embop{\operatorname{Res}}

\def\T{\Tubing}
\def\t{\tube}

\def\avg{\operatorname{avg}}

\def\coll{\operatorname{coll}}
\def\ex{\operatorname{ex}}
\def\Coll{\operatorname{Coll}}
\def\Ex{\operatorname{Ex}}

\def\tp{\tau_+}
\def\Adj{\operatorname{Adj}}
\def\Adjp{\Adj'}

\def\Pareq(#1,#2){#1^{\operatorname{min}}_{\supseteq #2}}

\def\ijP{P_{\t}^{\tp}}
\def\jiP{P_{\tp}^\t}

\def\ti{\t^\parr i}
\def\tpi{\tp^\parr i}

\def\tubes{{\mathbf{B}}}

\def\Coh{\operatorname{Coh}}
\def\tm{\t_-}
\def\Poincare{Poincar\'e\xspace}
\def\EQ[#1]{\overline{#1}}
\def\eq[#1]{\overline{#1}}
\def\lP{\prec_P}
\def\lp{\prec_P}
\def\gp{\succ_P}

\def\leqp{\preceq_P}
\def\lPa{\prec_{\Pa}}
\def\Paf{\Pa_f}
\def\lpaf{\prec_{\Paf}}
\def\lpa{\prec_{\Pa}}

\def\leqPa{\preceq_{\Pa}}
\def\leqpa{\preceq_{\Pa}}

\def\EQTubes(#1){[\Tubes(#1)]}

\def\Pdash{}
\def\Padash{$\Pa$-}
\def\Padash{}

\def\tbr[#1]{[#1]}

\def\RSZ{\R_{\Sigma=0}}

\def\Coo{\COrdop^\circ}
\def\Ohat{\hat{\Ordop}}
\def\proj{\pi}
\def\psz{\proj_{\Sigma=0}}

\def\Sbf{\mathbf{S}}
\def\Stel{\operatorname{Stel}}

\def\pbt{\parr\bt}
\def\adj{\operatorname{adj}}

\def\embop{\rho}
\def\res_#1{\embop_{#1}}
\def\Res#1#2{\embop_{#1}}
\def\ResP{\embop}
\def\ResPa{\tilde\embop}
\def\pan{^\parr n}
\def\pai{^\parr i}
\def\Perm{\Pi}
\def\PermB{\Pi^B}
\def\prodb{\bar\prod}

\def\Triples{\operatorname{Triples}}

\newcommand{\precdot}{\prec\mathrel{\mkern-5mu}\mathrel{\cdot}}

\def\pat{^\parr t}
\def\posetdash{$P$-}
\def\Posetdash{$P$-}
\def\aposetdash{$\Pa$-}

\def\piping{piping\xspace}
\def\pipings{pipings\xspace}
\def\Piping{Piping\xspace}

\def\pipe{pipe\xspace}
\def\pipes{pipes\xspace}

\def\Pdashpipe{$P$-pipe\xspace}
\def\Pdashpipes{$P$-pipes\xspace}
\def\Pdashpiping{$P$-piping\xspace}
\def\Pdashpipings{$P$-pipings\xspace}

\begin{document}
\numberwithin{equation}{section}

%
\title{$P$-associahedra}
\author{Pavel Galashin}
\address{Department of Mathematics, University of California, Los Angeles, CA 90095, USA}
\email{{\href{mailto:galashin@math.ucla.edu}{galashin@math.ucla.edu}}}
\thanks{The author was supported by an Alfred P. Sloan Research Fellowship and by the National Science Foundation under Grants No.~DMS-1954121 and No.~DMS-2046915.}
\date{\today}



\begin{abstract}
For each poset $P$, we construct a polytope $\Ass(P)$ called the \emph{$P$-associahedron}. Similarly to the case of graph associahedra, the faces of $\Ass(P)$ correspond to certain nested collections of subsets of $P$. The Stasheff associahedron is a compactification of the configuration space of $n$ points on a line, and we recover $\Ass(P)$ as an analogous compactification of the space of order-preserving maps $P\to\R$. Motivated by the study of totally nonnegative critical varieties in the Grassmannian, we introduce \emph{affine poset cyclohedra} and realize these polytopes as compactifications of configuration spaces of $n$ points on a circle. For particular choices of (affine) posets, we obtain associahedra, cyclohedra, permutohedra, and type $B$ permutohedra as special cases.
\end{abstract}

\maketitle

\begin{figure}[h]
\begin{tabular}{cccc}
\qquad\qquad&\begin{tikzpicture}[baseline=(Z.base)]
\coordinate(Z) at (0,-3);
\node(A) at (0,0){\includegraphics[width=0.15\textwidth]{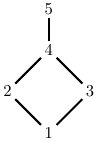}};
\end{tikzpicture}
 & \qquad\qquad&
\includegraphics[width=0.59\textwidth]{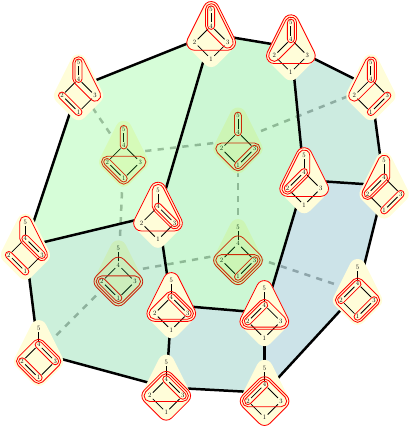} \\
\\
&$P$ & & $\Ass(P)$
\end{tabular}
  \caption{\label{fig:poset_ass_3d} A \posetdash associahedron.}
\end{figure}

\section{Introduction}
Polytopes arising from combinatorial data have been studied extensively in the recent decades. Some examples include \emph{order polytopes}~\cite{Stanley_two}, \emph{graph associahedra}~\cite{CaDe}, \emph{generalized permutohedra}~\cite{Pos_perm}, \emph{the associahedron}~\cite{Tamari,Stasheff,Haiman,Lee}, and \emph{the cyclohedron}~\cite{BoTa,Simion}. The latter two polytopes may be obtained as compactifications of configuration spaces of $n$ points on a line and on a circle, respectively; see e.g.~\cite{FuMa,AxSi,Kontsevich,Sinha,Gaiffi,LTV}.

 The goal of the present paper is to introduce a new class of polytopes called \emph{\posetdash associahedra} which combines the notions of graph associahedra and order polytopes in a natural way, and to show that these polytopes arise as compactifications of \emph{poset configuration spaces} of points on a line. We review these results in Sections~\ref{sec:intro:poset_ass}--\ref{sec:intro:compact}. We then introduce \emph{affine posets} and \emph{affine \aposetdash cyclohedra} in \cref{sec:intro:affine}. They correspond to compactifying \emph{affine poset configuration spaces} of points on a circle rather than on a line, and lead to applications to critical varieties~\cite{crit} which we pursue in a separate paper~\cite{crit_tnn}. %

\subsection{\posetdash associahedra}\label{sec:intro:poset_ass}
Let $(P,\leqp)$ be a finite connected poset with $|P|\geq2$. Recall from~\cite{Stanley_two} that the faces of the \emph{order polytope} of $P$ correspond\footnote{Stanley's construction of an order polytope only applies when $P$ has a minimal and a maximal element. In~\eqref{eq:intro:Oo_dfn}, we slightly modify his construction to include arbitrary connected posets.} to set partitions $\T$ of $P$ such that each $\t\in\T$ is a convex connected subset of $P$, and such that the directed graph $D_\T$ with vertex set $V(D_\T):=\T$ and edge set
\begin{equation}\label{eq:intro:acyclic}
  E(D_\T):=\{(\t,\t')\in\T^2\mid \t\cap \t'=\emptyset\text{ and }i\lp j \text{ for some $i\in\t$, $j\in\t'$}\}
\end{equation}
is acyclic. 
 Here a subset $\t\subseteq P$ is called \emph{convex} if having $i\lp j\lp k$ with $i,k\in\t$ implies $j\in \t$, and $\t$ is called \emph{connected} if the corresponding induced subgraph of the Hasse diagram of $P$ is connected. Let us say that two sets $A,B$ are \emph{nested} if either $A\subseteq B$ or $B\subseteq A$.
\begin{definition}
A \emph{\Pdashpipe} is a convex connected nonempty subset $\t\subseteq P$. A \emph{\Pdashpiping} is a collection $\T$ of \Pdashpipes such that any two sets $\t,\t'\in\T$ are either nested or disjoint, and such that the directed graph $D_\T$
given by~\eqref{eq:intro:acyclic} is acyclic.
\end{definition}

 When the poset $P$ is clear from the context, we refer to \Pdashpipes (resp., \Pdashpipings) simply as \emph{\pipes} (resp., \emph{\pipings}). We say that a \Pdash \pipe $\t$ is \emph{proper} if $1<|\t|<|P|$. A \Pdash \piping is \emph{proper} if it consists of proper \Pdash \pipes. Clearly, a subset of a proper \Pdash \piping is a proper \Pdash \piping. We let $\KAss(P)$ be an abstract simplicial complex whose vertices correspond to proper \Pdash \pipes, and whose simplices correspond to proper \Pdash \pipings.

\begin{theorem}[\Posetdash associahedron]\label{thm:poset_ass}
There exists a simplicial $(|P|-2)$-dimensional polytope $\Ass(P)^\ast$ whose boundary complex is isomorphic to $\KAss(P)$.
\end{theorem}
By definition, the \emph{\posetdash associahedron} $\Ass(P)$ is the polar dual of the polytope constructed in \cref{thm:poset_ass}. Thus $\Ass(P)$ is a simple polytope of dimension $|P|-2$ whose facets correspond to proper \Pdash \pipes and whose vertices correspond to maximal by inclusion proper \Pdash \pipings. See \cref{fig:poset_ass_3d,fig:ass_perm} for examples. 

We list some properties and examples of \posetdash associahedra in \cref{sec:prop}. For instance, similarly to other families of combinatorial polytopes (including permutohedra and associahedra), each face of $\Ass(P)$ is a product of smaller \posetdash associahedra. When $P$ is a chain, $\Ass(P)$ is combinatorially equivalent to the $(|P|-2)$-dimensional associahedron. When $P$ is a \emph{claw} (that is, $P$ contains a minimal element $\hat0$ and any two other elements of $P$ are incomparable), $\Ass(P)$ is combinatorially equivalent to the $(|P|-2)$-dimensional permutohedron. See \cref{fig:ass_perm} for two-dimensional examples. 

\begin{figure}

\scalebox{0.95}{
\begin{tabular}{cccc|cccc}
 \includegraphics[height=0.2\textwidth]{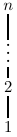}& \qquad&
 \includegraphics[width=0.32\textwidth]{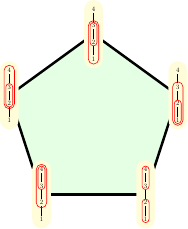}
& \qquad&\qquad & 
 \includegraphics[height=0.1\textwidth]{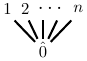}& 
\includegraphics[width=0.4\textwidth]{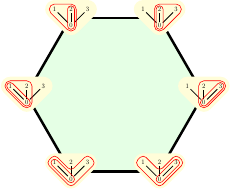}\\ &&&&&&& \\
$P$& & $\Ass(P)=$ associahedron & & & $P$ & $\Ass(P)=$ permutohedron
\end{tabular}
}
\caption{\label{fig:ass_perm} If $P$ is a chain (left) then $\Ass(P)$ is the associahedron. If $P$ is a claw (right) then $\Ass(P)$ is the permutohedron.}
\end{figure}

\begin{remark}
The set of \Pdash \pipes is not a \emph{building set} in the sense of~\cite{DCP,FeSt,Pos_perm} since the union of two \Pdash \pipes whose intersection is nonempty need not be a \Pdash \pipe. (It need not be convex.) Thus \posetdash associahedra are not special cases of graph associahedra or nestohedra. 
\end{remark}
\begin{remark}
Our notions of poset \pipes and \pipings differ from the well-studied notions of graph tubes and tubings~\cite{CaDe} in several ways. First, \Pdashpipes are assumed to be convex. Second, \Pdashpipes of size $1$ are not included in a proper \Pdashpiping. (In particular, when $P$ is a chain poset, \Pdashpipings may be more naturally viewed as bracketings; cf.~\cite[Figure~1(a)]{CaDe}.)  Third,  a graph tubing cannot contain two \emph{adjacent} graph tubes, i.e., two disjoint graph tubes whose union is a graph tube. We do not impose this restriction in the definition of \posetdash associahedra. Instead, we impose an acyclicity constraint~\eqref{eq:intro:acyclic} on \Pdashpipings.
\end{remark}
\begin{remark}
A different family of polytopes associated to posets was constructed in~\cite{DFRS}. We do not see any direct relation between the two constructions. It would be interesting to find the intersection of these two classes of polytopes.
\end{remark}
\begin{remark}
While we show that \posetdash associahedra $\Ass(P)$ exist as convex polytopes, we do not construct any explicit geometric realization of $\Ass(P)$ as a polytope with, say, integer vertex coordinates. Doing so remains an open problem.\footnote{Note added in 2023: This problem has now been solved in~\cite{Sack}.} Another interesting problem is to describe the $f$- and $h$-vectors of $\Ass(P)$ in terms of the combinatorics of $P$. 
\end{remark}

\begin{question}\label{que:Pasha}
In~\cite{LP_linear}, it was shown that graph associahedra of~\cite{CaDe} arise as dual cluster complexes of \emph{Laurent phenomenon algebras}~\cite{LP_LP}, which are certain generalizations of \emph{cluster algebras}~\cite{FZ}. Is there a relationship between \posetdash associahedra and dual cluster complexes of cluster algebras or of Laurent phenomenon algebras?
\end{question}
\noindent Another possible direction would be to relate \posetdash associahedra to $\tau$-tilting complexes of gentle algebras~\cite{PPPP}.

\subsection{Compactifications}\label{sec:intro:compact}
We explain how \posetdash associahedra may be obtained as compactifications of configuration spaces of points on a line. When $P$ is a chain, our construction recovers the case of Stasheff associahedra obtained as Axelrod--Singer compactifications~\cite{AxSi}; see also~\cite{FuMa}.

Recall that the \emph{order polytope}~\cite{Stanley_two} of $P$ is the space of order-preserving maps $P\to[0,1]$. We modify this construction to consider order-preserving maps $P\to\R$ instead. Let $\Sim$ be the group acting on $\R^P$ by \emph{rescalings} $\bx\mapsto \la\bx$ for $\la\in\Rtp$ and \emph{constant shifts} $\bx\mapsto \bx+\mu(1,1,\dots,1)$ for $\mu\in\R$. We let
\begin{equation}\label{eq:intro:Oo_dfn}
  \Oo(P):=\{\bx\in\R^P\mid x_i<x_j\text{ for all $i\lp j$}\}/\Sim
\end{equation}
denote the \emph{$P$-configuration space}. 
It is not hard to see (cf. \cref{sec:order_polyt}) that $\Oo(P)$ is naturally identified with the interior of a $(|P|-2)$-dimensional polytope denoted $\Ord(P)$. The faces of $\Ord(P)$ are indexed by \pipings $\T$ which are simultaneously set partitions of $P$.
 If $P$ is \emph{bounded}, i.e., contains a maximal and a minimal element, then $\Ord(P)$ is projectively equivalent to Stanley's order polytope; see \cref{rmk:Stanley}.

\begin{figure}
  \setlength{\tabcolsep}{0pt}
\makebox[1.0\textwidth]{
\scalebox{0.94}{
\begin{tabular}{c|c|c}
\begin{tikzpicture}[baseline=(Z.base)]
\coordinate(Z) at (0,-3);
\node(A) at (0,0){
\begin{tabular}{c}
  \includegraphics[width=0.15\textwidth]{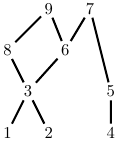}\ \strut \\ \\
(a) Poset $P$\\\\\hline \\
\includegraphics[width=0.15\textwidth]{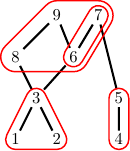}\ \strut
\end{tabular}
};
\end{tikzpicture}\hspace{-0.05in}
& \  \includegraphics[height=0.5\textwidth]{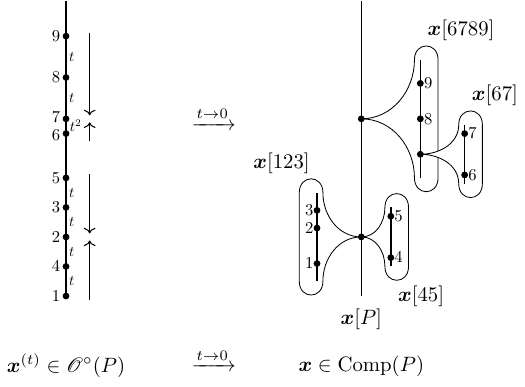} 
& \  \includegraphics[width=0.22\textwidth]{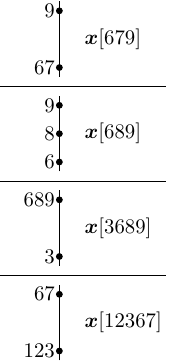}
\\
(b) \Piping $\T$ & (c) Taking a limit $\bx^{(t)}\xrightarrow{t\to0}\bx$ in $\Comp(P)$ & \ (d) $\bx[\t]$ for $\t\notin\T$
\end{tabular}
}
}
  \caption{\label{fig:comp} Defining the compactification $\Comp(P)$; see \cref{ex:comp}.}
\end{figure}

We will consider a certain compactification of $\Oo(P)$ which we first describe informally. See \cref{fig:comp} and \cref{ex:comp}. An element  $\bx\in\Oo(P)$ is a collection of $|P|$ points on a line satisfying the inequalities in~\eqref{eq:intro:Oo_dfn}. Allowing some (but not all, in view of the action of $\Sim$) of the points to collide, we obtain a point $\bx\in\Ord(P)$, which belongs to a face labeled by some set partition $\T_0=\{\t_1,\t_2,\dots,\t_m\}$ of $P$ into $m\geq2$ disjoint \Pdash \pipes. Thus all points in each \pipe $\t_j$ have collided, and moreover, it could be that all points in, say, $\t_1\sqcup\t_2$ have collided. During the collision, we keep track of the ``ratios of distances'' between all pairs of points inside each individual $\t_j$ (however, the distances between pairs of points in $\t_i\times\t_j$ for $i\neq j$ are ignored). In the limit, this gives a point $\bx\tbr[\t_j]\in \Ord(\Px[\t_j])$ for each $j=1,2,\dots,m$, where $\Px[\t_j]$ is treated as a connected subposet $(\t_j,\leqp)$ of $P$. We iterate this construction: the point $\bx\tbr[\t_j]$ belongs to some face of $\Ord(\Px[\t_j])$ labeled by a partition of $\t_j$ into disjoint \Pdash \pipes, so we record the distance ratios between pairs of points in each of those \pipes, etc. At the end, we obtain a collection $\T(\bx)$ of \Pdash \pipes which form a \Pdash \piping, and for each \Pdash \pipe $\t\in\T(\bx)$, we have a point $\bx[\t]\in\Ord(\Px[\t])$.

The non-rigorous part in the above paragraph is the notion of ``ratios of distances'' that we keep track of when the points collide. While such ratios are an essential ingredient in the definition of the Axelrod--Singer compactification~\cite{AxSi,Sinha,LTV}, we found that this approach cannot be directly applied to poset configuration spaces: see \cref{ex:bad_triples}. Instead, we utilize a new construction which we now describe formally. 

For a point $\bx\in\Oo(P)$ and a \pipe $\t$, let $\bx\tbr[\t]\in\Oo(\Px[\t])$ be the \emph{restriction} of $\bx$ to $\t$, i.e., the image of $\bx$ under the standard projection $\R^P\to\R^\t$. (This projection is $\Sim$-equivariant.) 
 Recall that $\Oo(\Px[\t])$ is identified with the interior of the order polytope $\Ord(\Px[\t])$. Consider the composite \emph{restriction map}
\begin{equation}\label{eq:intro:embP}
  \ResP: \Oo(P) \to \prod_{|\t|>1} \Ord(\Px[\t]),\quad \bx\mapsto (\bx\tbr[\t])_{|\t|>1},
\end{equation}
where the product is taken over \Pdash \pipes $\t$ satisfying $|\t|>1$. (This includes $\t=P$.)
\begin{definition}[\emph{$P$-compactification}]\label{dfn:intro:Comp}
Let $\Comp(P)$ denote the closure 
\begin{equation*}%
  \Comp(P):=\overline{\ResP(\Oo(P))}.
\end{equation*}
\end{definition}
Thus a point $\bx\in\Comp(P)$ is a collection $(\bx[\t])_{|\t|>1}\in\prod_{|\t|>1} \Ord(\Px[\t])$ of points in various order polytopes. We refer to the coordinates of $\bx[\t]$ as $(x_i[\t])_{i\in\t}$. We outlined above a recursive way to associate a proper \Pdash \piping $\T(\bx)$ to each such point $\bx\in\Comp(P)$; see \cref{dfn:bx_to_tubing} for further details. This endows $\Comp(P)$ with the structure of a stratified space, where the strata are indexed by proper \pipings.

\begin{example}\label{ex:comp}
Consider the poset $P$ in \figref{fig:comp}(a). For small $t>0$, the point $\bx^{(t)}$ shown on the left in \figref{fig:comp}(c) belongs to the $P$-configuration space $\Oo(P)$. When we take the limit as $t\to0$, we obtain a point $\bx\in\Comp(P)$, described as follows. The points $1,2,3,4,5$ collide, as do the points $6,7,8,9$, thus $\bx[P]\in\Ord[P]$ satisfies $x_1[P]=x_2[P]=\dots=x_5[P]$ and $x_6[P]=\dots=x_9[P]$. The set $\{1,2,3,4,5\}$ is a union of two \pipes, and the corresponding two points $\bx[123]\in\Ord(123)$ and $\bx[45]\in\Ord(45)$ are among those shown on the right in \figref{fig:comp}(c). Here we abbreviate $123=\{1,2,3\}$, etc. The two to one ratio of distances between the points $1,2$ and $2,3$ is encoded in the coordinates of $\bx[123]$. Similarly, the point $\bx[6789]\in\Ord(6789)$ satisfies $x_6[6789]=x_7[6789]$, but we have $x_6[67]<x_7[67]$. The \pipes $123, 45, 6789, 67$ form a proper \piping $\T:=\T(\bx)$ which labels (cf. \cref{dfn:bx_to_tubing}) the stratum of $\Comp(P)$ containing $\bx$. This \piping is shown in \figref{fig:comp}(b). By definition, to specify a point $\bx\in\Comp(P)$, one needs to specify a point $\bx[\t]\in\Ord(P)$ for \emph{any} \pipe $\t$, including the case $\t\notin\T$. Some of such points $\bx[\t]$ are shown in \figref{fig:comp}(d). As we explain in \cref{lemma:Coh_reconstruct}, it actually suffices to only specify the points $\bx[\t]$ for $\t\in\PT$.
\end{example}

\begin{theorem}\label{thm:intro:comp}
  There exists a stratification-preserving homeomorphism $\Ass(P)\xrasim\Comp(P)$.
\end{theorem}

\subsection{Affine \aposetdash cyclohedra}\label{sec:intro:affine}
We now describe affine versions of the above constructions, which have served as the original motivation for this work; see \cref{rmk:crit}.
\begin{definition}\label{dfn:intro:aff_poset}
An \emph{affine poset (of order $n\geq1$)}  is a poset $\Pa=(\Z,\leqPa)$ such that:
\begin{itemize}%
\item for all $i\in\Z$, $i\lPa i+n$;
\item for all $i,j\in\Z$, $i\leqPa j$ if and only if  $i+n\leqPa j+n$;%
\item $\Pa$ is \emph{strongly connected}: for all $i,j\in\Z$, we have $i\leqPa j+kn$ for some $k\geq0$.
\end{itemize}
\end{definition}
\noindent We denote the order of $\Pa$ by $|\Pa|:=n$.

A \emph{($\Pa$-)\pipe} is a convex connected nonempty subset of $\Z$ which either equals to $\Pa$ or contains at most one element in each residue class modulo $n$. Thus if $\t$ is a \Padash \pipe then so is $\t+dn$ for any $d\in\Z$, where we set $\t+dn:=\{i+dn\mid i\in\t\}$. We say that the \pipes $\t$ and $\t+dn$ are \emph{equivalent}, and let $\eq[\t]:=\{\t+dn\mid d\in\Z\}$ denote the equivalence class of $\t$. A collection of \pipes is called \emph{$n$-periodic} if it is a union of such equivalence classes.

A \emph{($\Pa$-)\piping} is an $n$-periodic collection $\T$ of \Padash \pipes such that any two \pipes in $\T$ are either nested or disjoint, and such that the directed graph $D_\T$ given by~\eqref{eq:intro:acyclic} is acyclic. 
 A \pipe $\t$ is called \emph{proper} if it satisfies $|\t|>1$ and $\t\neq \Pa$. A \piping is called \emph{proper} if it consists of proper \pipes. Observe that each \piping is a disjoint union of finitely many equivalence classes of \pipes. We let $\KCyc(\Pa)$ be an abstract simplicial complex whose vertices correspond to equivalence classes of proper \Padash \pipes, and whose simplices correspond to proper \Padash \pipings.
\begin{theorem}[Affine \aposetdash cyclohedron]\label{thm:aff_poset_cyc}
There exists a simplicial $(|\Pa|-1)$-dimensional polytope $\Cyc(\Pa)^\ast$ whose boundary complex is isomorphic to $\KCyc(\Pa)$.
\end{theorem}

\begin{figure}
  \makebox[1.0\textwidth]{
\scalebox{0.94}{
\ \ \quad\begin{tabular}{cccc|cccc}
 \includegraphics[height=0.4\textwidth]{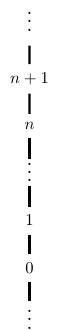}& \qquad&
  \includegraphics[width=0.35\textwidth]{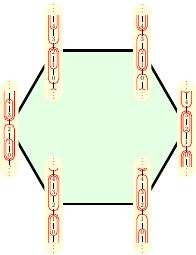} 
& \qquad&\qquad & 
 \includegraphics[height=0.3\textwidth]{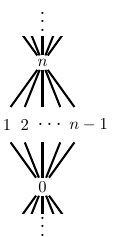}& \hspace{-0.3in}
\includegraphics[width=0.45\textwidth]{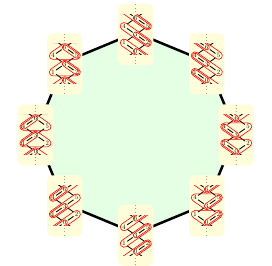}\\ &&&&&&& \\
$\Pa$& & $\Cyc(\Pa)=$ cyclohedron & & & $\Pa$ & $\Cyc(\Pa)=$ type $B$ permutohedron
\end{tabular}
}
}
  \caption{\label{fig:aff_cyc}If $\Pa$ is a circular chain (left) then $\Cyc(\Pa)$ is the cyclohedron. If $\Pa$ is a circular claw (right) then $\Cyc(\Pa)$ is the type $B$ permutohedron.}
\end{figure}

\noindent We define the \emph{affine \aposetdash cyclohedron} $\Cyc(\Pa)$ as the polar dual to $\Cyc(\Pa)^\ast$. See \cref{cor:aff_poset_cyc_properties} for a list of its properties. It is a simple $(|\Pa|-1)$-dimensional polytope whose vertices correspond to proper \pipings consisting of $|\Pa|-1$ equivalence classes of \pipes, and whose facets correspond to equivalence classes of proper \pipes. Each face of $\Cyc(\Pa)$ is a product of smaller \posetdash associahedra and affine \aposetdash cyclohedra. When $\Pa$ is a \emph{circular chain} shown in \figref{fig:aff_cyc}(left) (resp., a \emph{circular claw} shown in \figref{fig:aff_cyc}(right)), $\Cyc(\Pa)$ is combinatorially equivalent to the cyclohedron (resp., to the type $B$ permutohedron) of dimension $|\Pa|-1$. Since the cyclohedron is a type $B$ analog of the associahedron~\cite{Simion}, we may think of affine posets as type $B$ analogs of finite posets.

Finally, we explain how affine \aposetdash cyclohedra arise as compactifications. 
 Fix some constant $c\in\Rtp$. We identify points $\bx\in\R^{|\Pa|}$ with infinite sequences $\bxt=(\xt_i)_{i\in\Z}$ satisfying $\xt_{i+n}=\xt_i+c$ for all $i\in\Z$. Let the group $\R(1,1,\dots,1)$ act on $\R^{|\Pa|}$ by constant shifts. Set
 \begin{equation}\label{eq:intro:Orda}
\begin{aligned}%
  \Oao(\Pa)&:=\{\bx\in\R^{|\Pa|}/\R(1,1,\dots,1)\mid \xt_i<\xt_j\text{ for all $i\lpa j$}\},\\ 
\Orda(\Pa)&:=\{\bx\in\R^{|\Pa|}/\R(1,1,\dots,1)\mid \xt_i\leq\xt_j\text{ for all $i\leqpa j$}\}.
\end{aligned}
\end{equation}
We show in \cref{cor:Orda_nonempty} that $\Orda(\Pa)$ is a nonempty polytope of dimension $|\Pa|-1$. We call it the \emph{affine order polytope} of $\Pa$.

Given a point $\bx\in\Oao(\Pa)$ and a \pipe $\t$ with $|\t|>1$, we may still consider the restriction $\bx\tbr[\t]\in\Oo(\t)$ whose coordinates are given by $(\xt_i)_{i\in\t}$. (Recall that $\t=\Pa$ is considered a \pipe, in which case we set $\bx\tbr[\t]:=\bx$.) When two \pipes $\t,\t'$ are equivalent, we have $\bx\tbr[\t]=\bx\tbr[\t']$. 
We thus get a map
\begin{equation*}%
  \ResPa: \Oao(\Pa)\to\bar{\prod_{|\t|>1}} \Ord(\Pax[\t]),\quad \bx\mapsto(\bx\tbr[\t])_{|\t|>1}.
\end{equation*}
Here $\prodb_{|\t|>1} \Ord(\Pax[\t])$ is the set of points $(\bx[\t])_{|\t|>1}\in\prod_{|\t|>1} \Ord(\Pax[\t])$ satisfying $\bx[\t]=\bx[\t']$ whenever two \pipes $\t,\t'$ are equivalent. Thus essentially the product $\prodb_{|\t|>1} \Ord(\Pax[\t])$ is taken over finitely many equivalence classes $\eq[\t]$ of \pipes $\t$ satisfying $|\t|>1$, including the case $\t=\Pa$. For $\t\neq\Pa$, $\Ord(\Pax[\t])$ is the order polytope associated to the finite connected subposet $(\t,\leqpa)$ of $\Pa$.
 We consider the closure
\begin{equation}\label{eq:intro:Comp_Pa}
    \Comp(\Pa):=\overline{\ResPa(\Oao(\Pa))}.
\end{equation}
Similarly to the case of \posetdash associahedra, $\Comp(\Pa)$ admits a stratification into pieces indexed by proper \pipings.
\begin{theorem}
  There exists a stratification-preserving homeomorphism $\Cyc(\Pa)\xrasim\Comp(\Pa)$.
\end{theorem}

\begin{figure}
\begin{tabular}{ccc}
\includegraphics[width=0.45\textwidth]{fig_polyt/aff_claw_oct}  & &
\includegraphics[width=0.42\textwidth]{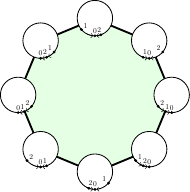} \\ \\
$\Cyc(\Pa)$ && $\Comp(\Pa)$
\end{tabular}
  \caption{\label{fig:circles} $\Cyc(\Pa)$ as a compactification of the $\Pa$-configuration space of points on a circle. See \cref{ex:circles}.}
\end{figure}
\begin{remark}
The quotient $\R/c\Z$ is homeomorphic to a circle $S^1$. Thus $\Oao(\Pa)$ may be considered as a configuration space of $|\Pa|$ points on $S^1$ (modulo global rotations of $S^1$) such that the points comparable in $\Pa$ are not allowed to pass through each other. When we take the closure in~\eqref{eq:intro:Comp_Pa}, we allow some (possibly all) of the points to collide. During the collisions, we keep track of the ratios of distances recursively as we did in \cref{sec:intro:compact}. In particular, when the points belonging to some \pipe $\t\neq\Pa$ collide, the relative distances between them are described by a point $\bx\tbr[\t]$ in the order polytope $\Ord(\Pax[\t])$ (as opposed to an affine order polytope). This is consistent with the fact that a circle is locally homeomorphic to a line.
\end{remark}

\begin{example}\label{ex:circles}
Suppose that $\Pa$ is a circular claw as in \figref{fig:aff_cyc}(right) of order $|\Pa|=3$. We may view $\Oao(\Pa)$ as the configuration space of three points labeled $0,1,2$ moving on a circle so that $1$ and $2$ can pass through each other, but neither $1$ nor $2$ can pass through $0$. %
 Consider the octagon in \figref{fig:circles}(right). Each vertex is labeled by a circle with points $0,1,2$ on it. We  view each such circular configuration as a limit as $t\to0$ of a family of configurations where the distance between $0$ and the closest point is $t^2$ while the distance between $0$ and the farthest point is $t$. In the limit as $t\to0$, it yields a point in $\Comp(\Pa)$ which corresponds to a vertex of $\Cyc(\Pa)$. This correspondence is illustrated in \cref{fig:circles}. 
\end{example}

\begin{remark}\label{rmk:crit}
Affine posets relevant to critical varieties are constructed as follows. Choose a permutation $f\in S_n$. Place $n$ vertices on a circle labeled $1,2,\dots,n$ in clockwise order. For each $s\in[n]:=\{1,2,\dots,n\}$, draw an arrow $s\to i$ whenever $i=f(s)$. The arrow starts slightly after $s$ and terminates slightly before $i$ in clockwise order; see \figref{fig:crit}(left). Assuming the resulting union of $n$ arrows is topologically connected, the affine poset $\Paf$ is defined as the $n$-periodic transitive closure of the relations
\begin{equation*}%
  i\lpaf j\lpaf i+n
\end{equation*}
for all $1\leq i<j\leq n$ such that the arrows $s\to i$ and $t\to j$ cross; see \figref{fig:crit}(right). Setting $c:=\pi$, the $\Paf$-configuration space $\Oao(\Paf)$ defined in~\eqref{eq:intro:Orda} coincides with the space $\THtp_f$ of \emph{$f$-admissible tuples} which parameterizes the \emph{critical cell $\Ctp_f$}; see~\cite[Definition~1.6]{crit}. As we show in~\cite[Theorem~4.1]{crit_tnn}, the affine \aposetdash cyclohedron $\Cyc(\Paf)$ admits a surjective continuous map onto the \emph{totally nonnegative critical variety} $\Ctnn_f$, defined as the closure of $\Ctp_f$ inside the Grassmannian.
\end{remark}

\begin{figure}
\begin{tabular}{ccc}
\includegraphics[width=0.25\textwidth]{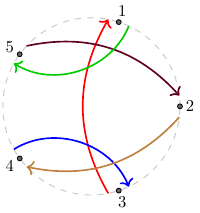}  & &
\includegraphics[width=0.25\textwidth]{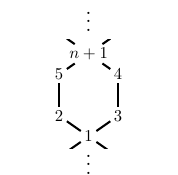} \\ \\
Strand diagram of $f\in S_n$ && Affine poset $\Paf$
\end{tabular}
  \caption{\label{fig:crit} Associating an affine poset $\Paf$ (right) to a strand diagram of a permutation $f\in S_n$ (left).}
\end{figure} 

\subsection*{Acknowledgments}
I am grateful to Dev Sinha for clarifications regarding some constructions in~\cite{Sinha}, which provided indispensable tools for the proof of \cref{thm:intro:comp} in \cref{sec:poset_compact}. I also thank Pasha Pylyavskyy for his initial contributions to~\cite{crit} and for bringing the relevant results of~\cite{LP_linear,CPY} to my attention; see \cref{que:Pasha} and \cref{rmk:Pasha}. Finally, I am grateful to the anonymous referees for their valuable suggestions, and in particular for their input regarding terminology choices such as \emph{\posetdash associahedra} and \emph{\pipings}.

\section{\Posetdash associahedra}

\subsection{Order cones and polytopes}\label{sec:order_polyt}
We start by collecting several simple results on order polytopes. %
 Let $P$ be a finite connected poset with $|P|\geq2$. First, rather than taking the quotient modulo the group $\Sim$ of rescalings and constant shifts, we would like to define $\Ord(P)$ as an explicit subset of $\R^P$. Let $\RSZ^P$ denote the subspace of $\R^P$ where the sum of coordinates is zero. Define a linear function $\al_P$ on $\R^P$ by
\begin{equation}\label{eq:al_dfn}
  \al_P(\bx):=\sum_{i\precdot_P j} x_j-x_i.
\end{equation}
Here the sum is taken over the covering relations $i\precdot_P j$ in $P$. We are ready to define the \emph{order cone $\COrd(P)$}, the \emph{order polytope $\Ord(P)$}, and their interiors:
\begin{align*}
\COrd(P)&:=\{\bx\in\RSZ^P\mid x_i\leq x_j\text{ for all $i\leqp j$}\}, & \Ord(P)&:=\{\bx\in\COrd(P)\mid \al_P(\bx)=1\};\\ 
\Coo(P)&:=\{\bx\in\RSZ^P\mid x_i< x_j\text{ for all $i\lp j$}\}, &\Oo(P)&:=\{\bx\in\Coo(P)\mid \al_P(\bx)=1\}.
\end{align*}
The definition and some basic properties of $\COrd(P)$ may be found e.g. in~\cite{PRW,JoSa}. Recall that a cone is called \emph{pointed} if it does not contain a line through the origin. Clearly, $\COrd(P)$ is a pointed polyhedral cone since for each $\bx\in \COrd(P)\setminus \{0\}$, we have $\al_P(\bx)>0$. Thus $\Ord(P)$ is a polytope of dimension $|P|-2$.

Next, we describe the faces of $\Ord(P)$. 
\begin{definition}
 A \emph{\piping partition} of $P$ is a \piping $\T$ which is simultaneously a set partition of $P$.
\end{definition}

Consider a point $\bx\in\Ord(P)$. Let $\tubes(\bx)$ be the collection of maximal by inclusion \pipes $\t$ such that we have $x_i=x_j$ for all $i,j\in\t$. Then $\tubes(\bx)$ is a \piping partition of $P$. Given an arbitrary \piping partition $\T$ of $P$, let
\begin{equation}\label{eq:tubes_dfn}
  \OFace(P,\T):=\{\bx\in\Ord(P)\mid \tubes(\bx)=\T\},
\end{equation}
and let $\OFacecl(P,\T)$ denote the closure of $\OFace(P,\T)$. For the coarsest \piping partition $\T=\{P\}$, let $\OFacecl(P,\T):=\emptyset$ denote the empty face of $\Ord(P)$. The following proposition is a straightforward extension of the results of~\cite{Stanley_two}.

\begin{proposition}\label{prop:Ord_faces}
The map $\T\mapsto \OFacecl(P,\T)$ is a bijection between \piping partitions $\Tubing$ of $P$ and faces of $\Ord(P)$. Face inclusion corresponds to refinement:
\begin{equation}\label{eq:Ord_faces}
  \OFacecl(P,\T)\subseteq\OFacecl(P,\T') \quad\Longleftrightarrow\quad \text{each $\tube'\in\Tubing'$ is contained in some $\tube\in\Tubing$}.
\end{equation}
The dimension of each face $\OFacecl(P,\T)$ equals $|\T|-2$.\qed
\end{proposition}
\begin{corollary}\ \label{cor:Ord_faces_properties}
\begin{enumerate}[label=\normalfont(\roman*)]
\item The vertices of $\Ord(P)$ are in bijection with partitions $P=I\sqcup F$ of $P$ into a connected nonempty order ideal $I$ and a connected nonempty order filter $F$.
\item The facets of $\Ord(P)$ are in bijection with the covering relations $i\precdot_P j$ in $P$.
\item Each face $\OFacecl(P,\T)$ of $\Ord(P)$ is itself an order polytope $\Ord(P/\T)$, where the quotient poset $P/\T$ is obtained from $P$ by identifying all elements of $P$ that belong to a single \pipe of $\T$.
\end{enumerate}
\end{corollary}

Let $\t\subseteq P$ be a \emph{non-singleton \pipe}, i.e., a \pipe satisfying $|\t|>1$. Recall that $\t$ is treated as a subposet $(\t,\leqp)$ of $P$. Given any set $A\supseteq \t$, define the following maps:
\begin{align*}%
 &\avg_\t:\R^A\to \R, \quad \bx\mapsto \frac1{|\t|}\sum_{i\in \t} x_i; &\quad &\psz^\t:\R^A\to\RSZ^\t,  \quad \bx\mapsto (x_i-\avg_\t(\bx))_{i\in \t};\\
  &\al_\t:\R^A\to\R,\quad \bx\mapsto \sum_{i,j\in\t:\ i\,\precdot_{P}\, j} x_j-x_i;&\quad &\res_\t:\R^A\dashrightarrow \R^\t,\quad \bx\mapsto \frac1{\al_\t(\bx)} \psz^\t(\bx).
\end{align*}
Here $\res_\t$ is a rational map defined on the subset of $\R^A$ where $\al_\t(\bx)\neq0$. 

\begin{remark}\label{rmk:suppress}
We suppress the dependence of the maps $\avg_\t,\psz^\t,\al_\t,\res_\t$ on $A$. Thus, for example, we have $\al_\t\circ\psz^\t=\al_\t$ as maps $\R^A\to\R$.
\end{remark}

The map $\res_P$ provides a homeomorphism between the $P$-configuration space defined in~\eqref{eq:intro:Oo_dfn} and the interior $\Oo(P)$ of $\Ord(P)$. More generally, suppose that $\t\subseteq\tp$ are non-singleton \pipes. Then we have a map
\begin{equation}\label{eq:psz_ord}
  \psz^{\Px[\t]}: \COrd(\Px[\tp])\to \COrd(\Px[\t]).
\end{equation}
The map $\res_\t:\COrd(\Px[\tp])\dashrightarrow\Ord(\Px[\t])$ is defined at all points $\bx\in\COrd(\Px[\tp])$ such that not all coordinates $\{x_i\mid i\in\t\}$ are equal. For the case $\tp=P$, we find that $\res_\t$ coincides with the map $\bx\mapsto\bx\tbr[\t]$ from \cref{sec:intro:compact}. Thus the map $\ResP$ in~\eqref{eq:intro:embP} extends to a map
\begin{equation}\label{eq:embP_Coo}
  \ResP:\Coo(P)\to \prod_{|\t|>1} \Ord(\Px[\t]),\quad \bx\mapsto (\Res{\t}{P}(\bx))_{|\t|>1}.
\end{equation}

\begin{remark}\label{rmk:Stanley}
Suppose $P$ is bounded, and denote by $\hat0,\hat1\in P$ its minimal and maximal elements. The order polytope $\Ohat(P)$, introduced by Stanley~\cite{Stanley_two}, is the set of all $\bx\in \R^P$ satisfying $x_{\hat0}=0$, $x_{\hat1}=1$, and $x_i\leq x_j$ for all $i\leqp j$. Letting $\alp_P(\bx):=x_{\hat1}-x_{\hat0}$, we see that the map $\psz^P$ provides an affine isomorphism between $\Ohat(P)$ and the polytope $\Ord'(P):=\{\bx\in\COrd(P)\mid \alp_P(\bx)=1\}$. Thus the polytopes $\Ohat(P)$ and $\Ord(P)$ are projectively equivalent. When $P$ is not bounded, it appears that the polytope $\Ord(P)$ has not been considered before.
\end{remark}

\subsection{Proof of Theorem~\ref{thm:poset_ass}}\label{sec:proof:poset_ass}
We use a variation of Lee's construction~\cite{Lee}. Our proof can be summarized as follows. Recall from \cref{prop:Ord_faces} that the faces of $\Ord(P)$ correspond to \piping partitions of $P$, and therefore the same holds for the polar dual $\Ord(P)^\ast$. For each proper \pipe $\t$, we have a face of $\Ord(P)^\ast$ corresponding to the partition 
\begin{equation}\label{eq:t-partition}
   \{\t\}\sqcup\{\{i\}\mid i\in P\setminus \tube\}.
\end{equation}
 We will show that $\Ass(P)^\ast$ is obtained from $\Ord(P)^\ast$ by performing \emph{stellar subdivisions} at all such faces. The order of stellar subdivisions is chosen so that the size of $\t$ is weakly decreasing along the way. Before we proceed with the proof, we consider an example of constructing the polar dual of the polytope $\Ass(P)$ from \cref{fig:poset_ass_3d}.

\begin{figure}
\makebox[1.0\textwidth]{
\def\Mscl{0.7}
  \setlength{\tabcolsep}{-10pt}
\scalebox{1.2}{
\begin{tabular}{ccc}
  \includegraphics[width=0.3\textwidth]{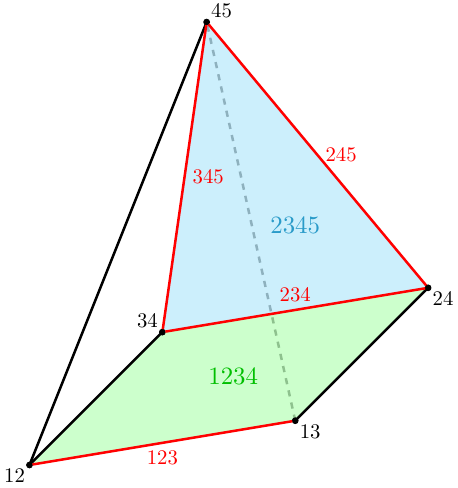} &  \includegraphics[width=0.3\textwidth]{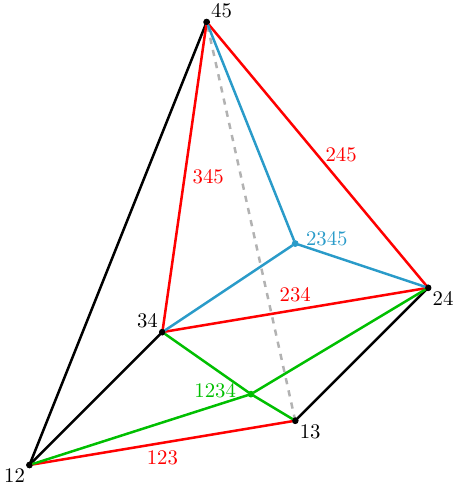}  &  \includegraphics[width=0.3\textwidth]{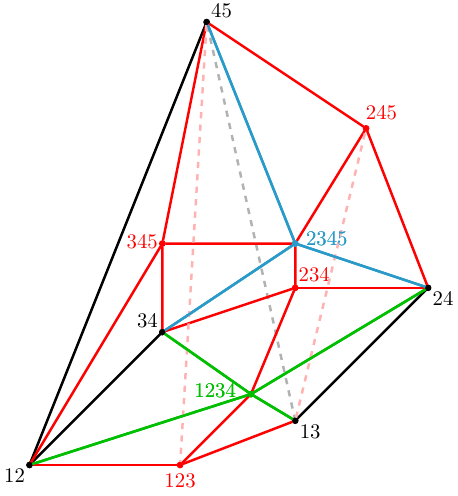} \\
\scalebox{\Mscl}{$\Mcal=\{P\}$} & \scalebox{\Mscl}{$\Mcal=\{P,1234,2345\}$\quad} &                                             \scalebox{\Mscl}{$\Mcal=\{P,1234,2345,123,234,245,345\}$}\\
\scalebox{\Mscl}{$\Ass_\Melt(P)^\ast=\Ord(P)^\ast$} && \scalebox{\Mscl}{$\Ass_\Melt(P)^\ast=\Ass(P)^\ast$}
\end{tabular}
}
}
  \caption{\label{fig:melt} Constructing the polytope $\Ass_\Melt(P)^\ast$ from $\Ord(P)^\ast$ inductively via stellar subdivisions. See \cref{ex:melt}.}
\end{figure}
\begin{example}\label{ex:melt}
Let $P=$
\scalebox{0.7}{\begin{tikzpicture}[scale=0.50,baseline=(C1.base)]
\coordinate (C1) at (0, 0);
\coordinate (C2) at (-1, 1);
\coordinate (C3) at (1, 1);
\coordinate (C4) at (0, 2);
\coordinate (C5) at (0, 3);
\node[ scale=0.80, inner sep=2pt] (N1) at (0, 0) {1};
\node[ scale=0.80, inner sep=2pt] (N2) at (-1, 1) {2};
\node[ scale=0.80, inner sep=2pt] (N3) at (1, 1) {3};
\node[ scale=0.80, inner sep=2pt] (N4) at (0, 2) {4};
\node[ scale=0.80, inner sep=2pt] (N5) at (0, 3) {5};
\draw[line width=1.00pt] (N1)--(N2);
\draw[line width=1.00pt] (N1)--(N3);
\draw[line width=1.00pt] (N2)--(N4);
\draw[line width=1.00pt] (N3)--(N4);
\draw[line width=1.00pt] (N4)--(N5);
\end{tikzpicture}} be the poset in \figref{fig:poset_ass_3d}(left). The polytope $\Ord(P)^\ast$ is shown in \figref{fig:melt}(left). Here and below we abbreviate $123:=\{1,2,3\}$, etc. The faces of $\Ord(P)^\ast$ correspond to \piping partitions of $P$, and each face of the form~\eqref{eq:t-partition} for some proper \pipe $\t$ is labeled by $\t$ in \figref{fig:melt}(left). For instance, the top left triangular face with vertices $\{12,34,45\}$ corresponds to the \piping partition $\{12,345\}$ which is not of the form~\eqref{eq:t-partition}, so we do not label this face in the figure. Next, we apply stellar subdivisions at all faces labeled by $4$-element \pipes, obtaining the polytope in \figref{fig:melt}(middle). The set $\Melt$, defined below, records the list of faces at which the subdivision has already been performed. 
We then apply stellar subdivisions at all faces labeled by $3$-element \pipes, obtaining the polytope in \figref{fig:melt}(right). Since $2$-element \pipes label the vertices of $\Ord(P)^\ast$, the corresponding stellar subdivisions do not change the polytope. The vertices of the resulting polytope in \figref{fig:melt}(right) are in bijection with proper \pipes, and a collection of vertices forms a face precisely when the corresponding \pipes form a \piping. Thus the polar dual of this polytope is combinatorially equivalent to $\Ass(P)$, as one can check by comparing \figref{fig:melt}(right) to \figref{fig:poset_ass_3d}(right).
\end{example}

We now explain the proof in detail. Suppose we are given a set $\Melt$ of \pipes such that for $\tube\subseteq \tube'$ with $\tube\in\Melt$, we have $\tube'\in\Melt$. We refer to the elements of $\Melt$ as \emph{melted \pipes}. A \pipe which does not belong to $\Melt$ is called \emph{frozen}.

A \piping $\Tubing$ satisfying $P\in\T$ is called \emph{$\Melt$-admissible} if
\begin{enumerate}[label=\normalfont(\alph*)]
\item\label{M_adm_fro} for each frozen \pipe $\tube\in\Tubing$, there is no $\tube'\in\Tubing$ such that $\tube'\subsetneq \tube$.
\item\label{M_adm_melt} for each melted \pipe $\tube\in\Tubing$,
 the maximal by inclusion \pipes $\tube'\in\Tubing$ satisfying $\tube'\subsetneq \tube$ form a \piping partition of $\tube$.
\end{enumerate}
Let $(\Adm(P;\Melt),\leq_\Melt)$ denote the poset of all $\Melt$-admissible \pipings, where $\Tubing\leq_\Melt \Tubing'$ if and only if $\T$ is obtained from $\T'$ by removing some melted \pipes and subdividing some frozen \pipes. More precisely, $\T\leq_\Melt \T'$ if
\begin{enumerate}
\item\label{tube_rel:frozen} for each frozen \pipe $\tube\in \Tubing$, there exists a frozen \pipe $\tube'\in\Tubing'$ satisfying $\tube\subseteq\tube'$, and
\item\label{tube_rel:melted} for each melted \pipe $\tube\in\Tubing$, we have $\tube\in \Tubing'$.
\end{enumerate}

Our proof will proceed by induction on $|\Melt|$, starting from the base case $\Melt=\{P\}$. For each set $\Melt$, we will introduce a polytope $\Ass_\Melt(P)^\ast$ whose boundary face lattice\footnote{By definition, the \emph{boundary face lattice}  includes all faces (in particular, the empty face) except for the polytope itself.} is isomorphic to $\Adm(P;\Melt)$. For each $\T\in\Adm(P;\Melt)$, we let $\FM_\T$ denote the corresponding face of $\Ass_\Melt(P)^\ast$. We will show that its dimension is given by
\begin{equation}\label{eq:FM_dim}
  \dim(\FM_\Tubing) = |P|+|\Tubing\cap \Melt|-|\Tubing\setminus \Melt|-2.
\end{equation}
For example, the minimal element of $\Adm(P;\Melt)$ consists of $P$ together with all singleton \pipes. (Throughout the entire induction process, the singleton \pipes stay frozen.) By~\eqref{eq:FM_dim}, the face corresponding to this minimal element has dimension $-1$ and thus is the empty face of $\Ass_\Melt(P)^\ast$. We encourage the reader to check that the face poset of the polytope in \figref{fig:melt}(middle) coincides with $\Adm(P;\Melt)$ for $\Mcal=\{P,1234,2345\}$.

Consider the base case $\Melt=\{P\}$. By definition, each $\Melt$-admissible \piping $\T$ contains $P$ together with a \piping partition of $P$ into frozen \pipes. The order relation $\leq_\Melt$ is given by coarsening, which is the opposite of~\eqref{eq:Ord_faces}. Thus we let $\Ass_\Melt(P)^\ast:=\Ord(P)^\ast$ be the polar dual of $\Ord(P)$. For example, maximal elements of $\Adm(P;\Melt)$ correspond to \pipings of the form $\T=\{P,I,F\}$ where $I$ (resp., $F$) is a nonempty order ideal (resp., order filter). By \cref{cor:Ord_faces_properties}, such \pipings are in bijection with the facets of $\Ass_\Melt(P)^\ast$. We check that~\eqref{eq:FM_dim} holds for the base case.

We now proceed with the induction step. Suppose we have constructed the polytope $\Ass_\Melt(P)^\ast$ as above for some set $\Melt$. Choose a maximal by inclusion frozen proper \pipe $\tube\notin \Melt$, and let $\Melt':=\Melt\sqcup\{\tube\}$. Set
\begin{equation*}%
  \Sbf_\tube:=\{P,\tube\}\sqcup\{\{i\}\mid i\in P\setminus \tube\}.
\end{equation*}
 Thus $\Scal_\t$ is an $\Mcal$-admissible \piping. Let $\FM_{\Sbf_\tube}$ be the corresponding face of $\Ass_\Melt(P)^\ast$. Our goal is to perform a stellar subdivision of $\Ass_\Melt(P)^\ast$ at the face $\FM_{\Sbf_\tube}$.

We give some background on stellar subdivisions; see e.g.~\cite[Exercise~3.0]{Ziegler} or~\cite[Section~2.1]{AdBe}. Let $Q$ be a polytope and $F\subsetneq Q$ be its face. Assume for simplicity that $Q$ contains the origin in its interior. Geometrically, a stellar subdivision $\Stel(Q,F)$ of $Q$ at the face $F$ is obtained by choosing a point $\bx$ in the relative interior of $F$ and setting 
\begin{equation*}%
  \Stel(Q,F):=\Conv(Q\cup \{(1+\eps)\bx\})
\end{equation*}
for some sufficiently small $\eps>0$. Combinatorially, the face poset of $\Stel(Q,F)$ is obtained from that of $Q$ via the following procedure:
\begin{enumerate}[label=\normalfont(\roman*)]
\item\label{stellar1} add a new vertex $\bx':=(1+\eps)\bx$;
\item\label{stellar2} remove all faces $F'$ of $Q$ containing $F$;
\item\label{stellar3} for each face $F'$ of $Q$ containing $F$ and each face $F''\subseteq F$ not containing $F$, add a new face $\Conv(F''\cup\{\bx'\})$ of dimension $\dim(F'')+1$.
\end{enumerate}

Going back to our proof, we let $F:=\FM_{\Sbf_\tube}$, $Q:=\Ass_\Melt(P)^\ast$, and $\Ass_{\Melt'}(P)^\ast:=\Stel(Q,F)$. Thus the face poset of $\Ass_{\Melt'}(P)^\ast$ is given  by steps~\ref{stellar1}--\ref{stellar3} above. Let us now compare $\Adm(P;\Melt)$ to $\Adm(P;\Melt')$ and show that $\Adm(P;\Melt')$ is obtained from $\Adm(P;\Melt)$ by applying analogs of steps~\ref{stellar1}--\ref{stellar3}.

\ref{stellar1}: $\Adm(P;\Melt')\setminus \Adm(P;\Melt)$ contains a \piping 
\begin{equation*}%
  \Sbf'_\tube:=\{P,\tube\}\sqcup\{\{i\}\mid i\in P\}.
\end{equation*}
 It corresponds to the new vertex $\bx'$ of $\Ass_{\Melt'}(P)^\ast$.

\ref{stellar2}: Let $\Tubing$ be an $\Melt$-admissible \piping such that $\Sbf_\tube\leq_\Melt \Tubing$, i.e., such that $\FM_{\Sbf_\tube}\subseteq \FM_{\Tubing}$. Since $\tube$ was a maximal by inclusion frozen \pipe, by~\eqref{tube_rel:frozen} we get that $\tube\in\Tubing$. In particular, $\Tubing$ is not $\Melt'$-admissible, thus the face $\FM_{\Tubing}$ is removed. Conversely, any $\T\in\Adm(P;\Melt)\setminus\Adm(P;\Melt')$ must contain $\t$.

\ref{stellar3}: Let $\Sbf_\tube\leq_\Melt \Tubing$ be as above. Any $\Tubing'\leq_\Melt\Tubing$ is obtained from $\Tubing$ by removing some melted \pipes and subdividing some frozen \pipes. Moreover, we have $\Sbf_\tube\not\leq_\Melt \Tubing'$ if and only if $\tube\notin \Tubing'$ (thus $\tube$ was among the subdivided frozen \pipes). 
In this case, we claim that $\T'':=\Tubing'\sqcup\{\tube\}$ is an $\Melt'$-admissible \piping. First, because $\T'$ contains a subdivision of $\t$, any two \pipes in $\T''$ are either nested or disjoint. Next, we need to show that the directed graph $D_{\T}''$ is acyclic. Suppose otherwise that $\t''_1\to\t''_2\to\cdots\to\t''_m\to\t''_{m+1}=\t''_1$ is a directed cycle in $D_{\T}''$. For each $j\in[m]$, let $\t_j\in\T$ be the minimal by inclusion \pipe containing $\t''_j$. We see that $\t''_j\in\Melt'$ if and only if $\t_j\in\Melt'$, in which case $\t_j=\t''_j$. Let $D_\T$ be the directed graph obtained from $\T$ via~\eqref{eq:intro:acyclic}. Let $j\in[m]$. If $\t_j\cap \t_{j+1}=\emptyset$ then $(\t_j,\t_{j+1})$ is an edge in $D_\T$. Otherwise, $\t_j$ and $\t_{j+1}$ must be nested, say, $\t_j\subseteq\t_{j+1}$. Since $\t''_j\cap \t''_{j+1}=\emptyset$, we cannot have $\t''_{j+1}=\t_{j+1}$, so $\t_{j+1}\notin \Melt'$ is frozen, and therefore $\t_j=\t_{j+1}$ by \ref{M_adm_fro}. Because $\T'$ is itself a \piping, we must have $\t''_i=\t=\t_i$ for some $i\in[m]$. Therefore not all \pipes $\t_j$ are equal to each other. We arrive at a directed cycle in $D_\T$, a contradiction. We have shown that $\T''$ is a \piping. Finally, because $\t\in\Melt'$ is subdivided in $\T'$, the \piping $\T''$ is $\Melt'$-admissible. This way, we obtain all $\Melt'$-admissible \pipings containing $\tube$.

We let $\FMp_{\T''}$ be the face $\Conv(\FM_{\T'}\cup\{\bx'\})$ of $\Stel(Q,F)$. We find that $\dim\FMp_{\T''}=\dim\FM_{\T'}+1$, which is consistent with~\eqref{eq:FM_dim} since $\T''=\T'\sqcup\{\t\}$ and $\t\in\Melt'$. 
 Any $\Melt'$-admissible \piping not containing $\tube$ is already $\Melt$-admissible. This exactly parallels the description in step~\ref{stellar3}. We have shown that $\Adm(P;\Melt')$ is the boundary face lattice of $\Ass_{\Melt'}(P)^\ast$, completing the induction step.

We continue this process until $\Melt$ contains all proper \pipes. Then every $\Melt$-admissible \piping contains $P$ and all singleton \pipes. Removing them, we obtain an order-preserving bijection between $\Adm(P;\Melt)$ and the poset of proper \pipings ordered by inclusion. Thus the boundary face poset $\Adm(P;\Melt)$ of $\Ass_{\Melt}(P)^\ast$ is isomorphic to the face poset of the simplicial complex $\KAss(P)$ in \cref{thm:poset_ass}. \qed

\subsection{Properties of \posetdash associahedra}\label{sec:prop}
Recall that a poset $P$ is called a \emph{chain} if its covering relations are $1\precdot_P 2\precdot_P\cdots \precdot_P n$, and $P$ is called a \emph{claw} if its covering relations are $\hat0\precdot_P1,\hat0\precdot_P2,\dots,\hat0\precdot_Pn$. 
In the following result, we identify two polytopes if they are combinatorially equivalent.
\begin{corollary}\label{cor:poset_ass_properties}
Let $P$ be a finite connected poset with $|P|\geq2$.
\begin{enumerate}[label=\normalfont(\roman*)]
\item\label{PApr:simple} $\Ass(P)$ is a simple polytope of dimension $|P|-2$.
\item\label{PApr:flag} Its polar dual $\Ass(P)^\ast$ is simplicial, but in general not flag.
\item\label{PApr:faces_dim} For each proper \piping $\T$, the corresponding face of $\Ass(P)$ has dimension $|P|-|\T|-2$.
\item\label{PApr:vert} The vertices of $\Ass(P)$ are in bijection with proper \pipings of size $|P|-2$.
\item\label{PApr:facets} The facets of $\Ass(P)$ are in bijection with proper \pipes.
\item\label{PApr:faces} Each face of $\Ass(P)$ is a product of \posetdash associahedra.
\item\label{PApr:chain} When $P$ is a chain, $\Ass(P)$ is the $(|P|-2)$-dimensional associahedron.
\item\label{PApr:claw} When $P$ is a claw, $\Ass(P)$ is the $(|P|-2)$-dimensional permutohedron.
\end{enumerate}
\end{corollary}
\begin{proof}
Most of these properties are simple consequences of the definitions and \cref{thm:poset_ass}. We comment on some of them.

\begin{figure}
\includegraphics[width=0.1\textwidth]{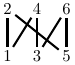}
  \caption{\label{fig:intr}A poset $P$ such that $\Ass(P)^\ast$ is not a flag simplicial complex.}
\end{figure}

\ref{PApr:flag}: Consider the poset $P$ in \cref{fig:intr}. The proper \pipes $\{1,2\}$, $\{3,4\}$, $\{5,6\}$ correspond to three vertices of $\Ass(P)^\ast$ such that any two of them form an edge of $\Ass(P)^\ast$. However, $\T:=\{\{1,2\},\{3,4\},\{5,6\}\}$ is not a \piping since the graph $D_\T$ contains a directed cycle. Thus these three vertices do not form a $2$-dimensional face of $\Ass(P)^\ast$, and therefore the boundary of $\Ass(P)^\ast$ is not a flag simplicial complex.

\ref{PApr:faces}: Consider a proper \piping $\T$. For each $\t\in\PT$, consider the quotient $\t/\T[\t]$ of the poset $\t$ obtained by identifying all elements which belong to some $\tm\in\T$ satisfying $\tm\subsetneq \t$. Then the face of $\Ass(P)$ corresponding to $\T$ is combinatorially equivalent to the product $\prod_{\t\in\PT}\Ass(\t/\T[\t])$ of such quotient \posetdash associahedra.

\ref{PApr:chain}: Let $n:=|P|$. Recall that the faces of the $(n-2)$-dimensional associahedron are in bijection with plane rooted trees with $n$ leaves, where the root has degree $\geq2$. Face closure relations correspond to edge contractions in such trees. In view of \cref{dfn:tree} below, it follows that when $P$ is a chain, plane rooted trees with $n$ leaves are in bijection with proper \pipings. Explicitly, we may assume that each plane tree is embedded in the upper half plane with the leaves lying on the $x$-axis. Labeling the leaves $1,2,\dots,n$ from left to right, each non-leaf vertex $v$ gives rise to a \pipe $\t_v$ consisting of the labels of its descendant leaves. The collection of $\t_v$ over all non-leaf vertices $v$ other than the root of the tree gives a proper \piping. Clearly, each proper \piping arises from a unique such plane rooted tree.

\ref{PApr:claw}: Let $n:=|P|-1$. Label the elements of $P\setminus\{\hat0\}$ by $1,2,\dots,n$ as in \figref{fig:ass_perm}(right). Recall that the $(n-1)$-dimensional \emph{permutohedron} $\Perm_n$ is the convex hull of all vectors obtained from $(1,2,\dots,n)$ by permuting the coordinates. The faces of $\Perm_n$ are in bijection with \emph{ordered set partitions} $(B_1,B_2,\dots,B_k)$, where $[n]=B_1\sqcup B_2\sqcup\cdots\sqcup B_k$ and each $B_i$ is nonempty. For each $i\in[k]$, let $\t_i:=\{\hat0\}\sqcup B_1\sqcup\cdots\sqcup B_i$. We obtain a proper \piping $\T:=\{\t_i\mid i\in[k]\}$, and the resulting map gives the desired order-preserving bijection.
\end{proof}

\section{\Posetdash associahedra as compactifications}\label{sec:poset_compact}
We develop some further properties of compactifications introduced in \cref{sec:intro:compact} and prove \cref{thm:intro:comp}. Before we proceed with the proof, we demonstrate a problem that arises when extending the definition of Axelrod--Singer compactifications to poset configuration spaces.  The standard approach~\cite{AxSi,Sinha,LTV} when $P$ is a chain is to consider a family of functions
\begin{equation*}%
  d_{i,j,k}:\Oo(P)\to[0,\infty],\quad d_{i,j,k}(\bx):=\frac{|x_i-x_j|}{|x_i-x_k|}
\end{equation*}
for all triples $i,j,k\in P$ of distinct elements. The Axelrod--Singer compactification is then essentially the closure of the image of $\Oo(P)$ inside the corresponding ${|P|\choose 3}$-dimensional space.\footnote{For arbitrary manifolds, one needs to include other functions keeping track of the coordinates $x_i$ and the directions of the unit vectors $\frac{x_i-x_j}{|x_i-x_j|}$ when the points $x_i$ and $x_j$ collide. In the $1$-dimensional case, ignoring these extra  functions does not alter the resulting compactifications.} In order to apply a similar construction to an arbitrary poset $P$, one would expect to fix some set $\Triples(P)$ consisting of some of the ${|P|\choose 3}$ triples $(i,j,k)$, and then define $\Comp(P)$ to be the closure of the image of $\Oo(P)$ inside the corresponding $|\Triples(P)|$-dimensional space. The following example demonstrates that this is impossible.
\begin{example}\label{ex:bad_triples}
\begin{figure}
  \setlength{\tabcolsep}{12pt}
\begin{tabular}{cccc}
\includegraphics[width=0.1\textwidth]{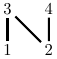} & 
\includegraphics[width=0.3\textwidth]{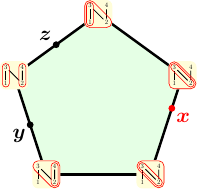}&
\includegraphics[width=0.4\textwidth]{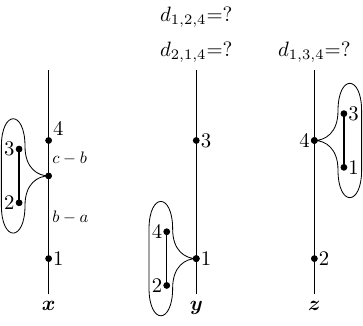}\\
$P$ & $\Ass(P)$ & 
\end{tabular}
  \caption{\label{fig:N_shaped} The functions $d_{i,j,k}$ needed to determine the ratio ${(c-b):(b-a)}$ cannot be extended to the whole boundary of $\Ass(P)$. See the proof of \cref{cor:poset_ass_properties}\ref{PApr:claw}.}
\end{figure}

Let $P$ be the \emph{$N$-shaped poset} with relations $1\lp 3 \gp 2\lp 4$. Thus $\Ass(P)$ is a pentagon; see \cref{fig:N_shaped}. Let $\bx\pat\in\Oo(P)$ be a configuration of four points on a line. We will use the definition~\eqref{eq:intro:Oo_dfn} of $\Oo(P)$ as a subset of $\R^P/\Sim$. Consider a point $\bx\in\Ass(P)$ obtained as the limit $x\pat_1\to a$, $x\pat_2,x\pat_3\to b$, $x\pat_4\to c$. Letting $a< b< c$ vary, we obtain a $1$-dimensional face of $\Ass(P)$. Thus there should be a triple $(i,j,k)\in\Triples(P)$ such that $d_{i,j,k}$ allows one to recover the ratio $(c-b):(b-a)$. The set of such triples $(i,j,k)$, modulo the symmetry of $P$ swapping $1\leftrightarrow 4$ and $2\leftrightarrow3$, and modulo swapping $j$ and $k$ in $d_{i,j,k}$, consists of $(1,2,4)$, $(2,1,4)$, and $(1,3,4)$. Suppose $(1,2,4)\in\Triples(P)$ or $(2,1,4)\in\Triples(P)$. Then consider a different limit where $x\pat_1,x\pat_2,x\pat_4\to 0$ and $x\pat_3\to 1$. This limit should yield a single point $\by\in \Ass(P)$. However, depending on how $x\pat_1,x\pat_2,x\pat_4$ approach $0$, the ratios $d_{1,2,4}(\bx\pat)$ and $d_{2,1,4}(\bx\pat)$ may converge to any numbers in $[0,\infty]$. Thus $(1,2,4),(2,1,4)\notin\Triples(P)$. Similarly, if $(1,3,4)\in\Triples(P)$ then we consider a limit where $x\pat_1,x\pat_3,x\pat_4\to1$, $x\pat_2\to 0$. This should yield a single point $\bz\in\Ass(P)$ but the ratio $d_{1,3,4}(\bx\pat)$ again may converge to any number depending on the way we take the limit. We arrive at a contradiction. See \cref{fig:N_shaped}.

The same problem arises if we consider more general distance ratio functions $d'_{i,j,k,\ell}(\bx):=\frac{|x_i-x_j|}{|x_k-x_\ell|}$. For instance, the above ratio $(c-b):(b-a)$ may be recovered from $d'_{1,3,2,4}$. However, considering a limit  $x\pat_1,x\pat_3\to 1$, $x\pat_2,x\pat_4\to0$ corresponding to a vertex of $\Ass(P)$, we again conclude that $d'_{1,3,2,4}$ cannot be used in the construction.
\end{example}

\subsection{Coherent collections}\label{sec:coh}
Our first goal is to describe which elements of $\prod_{|\t|>1} \Ord(\Px[\t])$ belong to $\Comp(P)$. For that, we will introduce the notion of a \emph{coherent collection}. Recall from~\eqref{eq:psz_ord} that for any \pipes $\t\subseteq\tp$ with $|\t|>1$, the map $\psz^\t$ gives a projection $\COrd(\tp)\to\COrd(\t)$.
\begin{definition}\label{dfn:coh}
An element $\bx\in \prod_{|\t|>1} \Ord(\Px[\t])$ is called \emph{coherent} if
\begin{equation}\label{eq:coh}
\text{for any $\t\subsetneq\t_+$ with $|\t|>1$, there exists $\la\in\Rtnn$ such that }\psz^\t(\bx\tbr[\tp])=\la\bx\tbr[\t].
\end{equation}
\end{definition}
\noindent We let $\Coh(P)$ denote the set of points $\bx\in \prod_{|\t|>1} \Ord(\Px[\t])$ satisfying~\eqref{eq:coh}. We will see later in \cref{prop:Comp=Coh} that $\Coh(P)=\Comp(P)$.

\begin{remark}\label{rmk:blow_up}
For  $(\by,\bz)\in\C^d\times(\C^d\setminus\{\bnull\})$, the condition that there exists some $\la\in\C$ satisfying $\by=\la\bz$ cuts out a subvariety of $\C^d\times (\C^d\setminus\{\bnull\})$ defined by equations $y_iz_j=y_jz_i$ for all $i, j\in[d]$. This construction is closely related to the classical notion of a \emph{blow-up} in algebraic geometry; see e.g.~\cite[page 28]{Hartshorne}. Thus the space $\Coh(P)$ may be considered a \emph{polytope-theoretic blow-up} of $\Ord(P)$ along the collection of faces indexed by \piping partitions of the form~\eqref{eq:t-partition}. We note that there is a well-known connection between blow-ups of toric varieties and stellar subdivisions of the associated polytopes; see e.g.~\cite[Section~1.7]{Oda}. It would be interesting to find some family of algebraic varieties reflecting the combinatorics of \posetdash associahedra.
\end{remark}

\subsection{A cell decomposition}\label{sec:cell_decn}
Given a point $\bx\in\Coh(P)$ and a non-singleton \pipe $\t$, we have a point $\bx[\t]\in\Ord(\t)$. We may therefore consider the corresponding \piping partition $\tubes(\bx[\t])$ of $\t$ defined in \cref{sec:order_polyt}. 
\begin{definition}\label{dfn:bx_to_tubing}
Let $\bx\in\Coh(P)$. Let $\HAT(\bx)$ be the smallest collection of \pipes such that
\begin{itemize}
\item $\HAT(\bx)$ contains $P$;
\item for each non-singleton $\t\in\HAT(\bx)$, $\HAT(\bx)$ also contains all \pipes in $\tubes(\bx\tbr[\t])$.
\end{itemize}
In particular, $\HAT(\bx)$ contains $P$ and all singleton \pipes. We let $\T(\bx)$ denote the set of proper \pipes in $\HAT(\bx)$. For an arbitrary proper \piping $\T$, we let $\HAT$ be obtained from $\T$ by adding $P$ and all singleton \pipes.
\end{definition}
\begin{definition}\label{dfn:tree}
Let $\HAT$ be a collection of \pipes containing $P$ and all singleton \pipes, such that any two \pipes in $\HAT$ are either nested or disjoint. Then $\HAT$ has the following structure of a rooted tree. The \pipe $P\in\HAT$ is the root, and the singleton \pipes are the leaves. For each non-singleton \pipe $\t\in\HAT$, the set $\HAT[\t]$ of its children consists of all maximal by inclusion \pipes $\tm\in\HAT(\bx)$ satisfying $\tm\subsetneq \t$. 
\end{definition}
\begin{lemma}\label{le:bx_to_tubing}
For any $\bx\in\Coh(P)$, $\T(\bx)$ is a proper \piping.
\end{lemma}
\begin{proof}
It is clear that any two \pipes in $\HAT:=\HAT(\bx)$ are either nested or disjoint. We need to show that the directed graph $D_{\T}$ is acyclic. Suppose otherwise that $\t'_1\to\t'_2\to\cdots\to\t'_m\to\t'_{m+1}=\t'_1$ is a cycle in $D_{\T}$. Let $\tp\in\HAT$ be the lowest common ancestor (cf. \cref{dfn:tree}) of $\t'_1,\t'_2,\dots,\t'_m$. For each $j\in[m]$, let $\t_j$ be the child of $\tp$ containing $\t'_j$. Thus the \pipes $\t_1,\t_2,\dots,\t_m$ are not all equal to each other. The children of $\tp$ in $\HAT$ form a \piping partition $\HAT[\tp]$ of $\tp$ equal to $\tubes(\bx[\tp])$. Thus for each $j\in[m]$, either $\t_j=\t_{j+1}$ or $\t_j\cap \t_{j+1}=\emptyset$, in which case $\t_j\to\t_{j+1}$ is an edge of $D_{\T}$. We have therefore found a cycle in $D_{\T}$ consisting of children of $\tp$, which contradicts the fact that they form a \piping partition of $\tp$.
\end{proof}

Our next goal is to show that any point $\bx\in\Coh(P)$ is completely determined by the points $\bx[\t]$ for all  $\t\in\HAT(\bx)$ (as opposed to all \pipes $\t$) satisfying $|\t|>1$; cf. \figref{fig:comp}(c,d) and \cref{ex:comp}.
\begin{definition}\label{dfn:Pareq}
Given an arbitrary subset $A\subseteq P$ and a \piping $\HAT$ with $P\in\HAT$, let $\Pareq(\HAT,A)$ be the minimal by inclusion \pipe $\t\in \HAT$ satisfying $\t\supseteq A$.
\end{definition}
\begin{lemma}\label{lemma:Coh_reconstruct}
Let $\bx\in\Coh(P)$ and $\HAT:=\HAT(\bx)$. Let $\t$ be any non-singleton \pipe, and let $\tp:=\Pareq(\HAT,\t)$. Then
\begin{equation*}%
  \al_\t(\bx[\tp])>0 \quad\text{and}\quad  \bx[\t]=\Res{\t}{\tp}(\bx[\tp]).
\end{equation*}
\end{lemma}
\begin{proof}
  Because $\tp\in\HAT$ is minimal by inclusion containing $\t$, we see that $\t$ is not contained in any \pipe in the \piping partition $\HAT[\tp]=\tubes(\bx[\tp])$. In particular, not all coordinates $\{x_i[\tp]\mid i\in\t\}$ are equal. Thus $\al_\t(\bx[\tp])>0$ and $\psz^\t(\bx[\tp])\neq\bnull$. By~\eqref{eq:coh}, we have $\la\bx[\t]=\psz^\t(\bx[\tp])$, and since the right hand side is nonzero, we have $\la>0$. It follows that $\la=\al_\t(\bx[\tp])$, thus $\bx[\t]=\Res{\t}{\tp}(\bx[\tp])$.
\end{proof}

\begin{proposition}\label{prop:Comp=Coh}
We have $\Comp(P)=\Coh(P)$.
\end{proposition}
\begin{proof}
First,~\eqref{eq:coh} is satisfied for all points in  $\ResP(\Oo(P))$. We explained in \cref{rmk:blow_up} that~\eqref{eq:coh} is described by polynomial equations and thus it is satisfied for the points in the closure $\Comp(P)$ of $\ResP(\Oo(P))$. Therefore $\Comp(P)\subseteq\Coh(P)$.

Conversely, let $\bx\in\Coh(P)$ and $\T:=\T(\bx)$. The following argument is borrowed from~\cite[Section~3.4]{Sinha}. Choose a vector $\bt=(t_\t)_{\t\in\T}\in\Rtp^{\T}$ such that
\begin{equation}\label{eq:coh_suff_small}
  0<t_\t\ll 1 \quad\text{for all $\t\in\T$,\quad  and}\quad t_{\tm}\ll t_\t \quad\text{for all $\tm,\t\in\T$ such that $\tm\subsetneq\t$}.
\end{equation}
Define a point $\by^\pbt\in \R^P$ by
\begin{equation*}%
  y_i^\pbt:=x_i\tbr[P]+\sum_{\t\in\T:\ i\in\t} t_\t x_i\tbr[\t], \quad\text{for all $i\in P$}.
\end{equation*}
It is easy to see that for $\bt$ sufficiently small satisfying~\eqref{eq:coh_suff_small}, we have $\by^\pbt\in\COo(P)$. Let
\begin{equation*}%
  \bz^\pbt:= \ResP \left(\by^\pbt\right) \quad \in \prod_{|\t|>1} \Ord(\Px[\t]);
\end{equation*}
cf.~\eqref{eq:embP_Coo}. We claim that $\lim_{\bt\to\bnull} \bz^\pbt=\bx$ inside $\prod_{|\t|>1} \Ord(\Px[\t])$,
 where the limit is taken in the above regime~\eqref{eq:coh_suff_small}. In other words, we need to show that $\lim_{\bt\to\bnull} \bz^\pbt[\t]=\bx[\t]$ for each non-singleton \pipe $\t$. This is clear for $\t=P$. Suppose next that $\t\in\T$. Define a point $\bx^\pbt[\t]\in\Coo(\t)$ by
\begin{equation*}%
  x^\pbt_i[\t]:=x_i[\t]+\sum_{\t_-\in\T:\ i\in\t_-\subsetneq\t} \frac{t_{\tm}}{t_\t} x_i\tbr[\t_-] \quad\text{for $i\in\t$}.
\end{equation*}
Thus $\bz^\pbt[\t]=\frac1{\al_\t(\bx^\pbt[\t])}\bx^\pbt[\t]$. By~\eqref{eq:coh_suff_small}, we have $\bx^\pbt[\t]\to\bx[\t]$ inside $\COrd(\t)$ as $\bt\to\bnull$. Thus $\al_\t(\bx^\pbt[\t])\to 1$ and  $\bz^\pbt[\t]\to\bx[\t]$ as $\bt\to\bnull$. We have shown the result for $\t\in\T$. 
For any proper \pipe $\t\notin\T$, the result follows by \cref{lemma:Coh_reconstruct}: for $\tp:=\Pareq(\HAT,\t)$, the map $\Res{\t}{\tp}:\COrd(\tp)\dashrightarrow\Ord(\t)$ is continuous where it is defined, and its domain of definition includes the points $\bx[\tp]$ and $\bz^\pbt[\tp]$.
\end{proof}

\begin{definition}
Given a proper \piping $\T$, let
\begin{equation*}%
  \D_\T:=\{\bx\in\Comp(P)\mid \T(\bx)=\T\}.
\end{equation*}
\end{definition}

Recall from \cref{dfn:tree} that for a proper \piping $\T$ and a \pipe $\t\in\HAT$, we denote by $\HAT[\t]$ the \piping partition of $\t$ consisting of all children of $\t$ in the rooted tree $\HAT$.
\begin{proposition}\label{prop:DT_cell}
For each proper \piping $\T$, we have a homeomorphism
\begin{equation*}%
  \DT\cong \prod_{\t\in\PTubing} \OFace(\Px[\t],\HAT[\t]).
\end{equation*}
\end{proposition}
\begin{proof}
Let $\bx\in\DT$. By \cref{dfn:bx_to_tubing}, we have 
\begin{equation}\label{eq:bx_prod_OFace}
  (\bx\tbr[\t])_{\t\in\PTubing} \in \prod_{\t\in\PTubing} \OFace(\Px[\t],\HAT[\t]).
\end{equation}
We claim that the map $\bx\mapsto  (\bx\tbr[\t])_{\t\in\PTubing}$ is a homeomorphism. To describe the inverse of this map, choose a point $(\bx\tbr[\t])_{\t\in\PTubing}$ as in~\eqref{eq:bx_prod_OFace}. Take any non-singleton \pipe $\t$ and let $\tp:=\Pareq(\HAT,\t)$. By \cref{lemma:Coh_reconstruct}, we must set $\bx[\t]:=\Res{\t}{\tp}(\bx\tbr[\tp])$. This defines a point $\bx\in\prod_{|\t|>1} \Ord(\Px[\t])$. We claim that $\bx\in\Coh(P)$, i.e., that it satisfies~\eqref{eq:coh}.

Let $\t\subsetneq\tp$ be arbitrary \pipes with $|\t|>1$. Our goal is to show that $\psz^\t(\bx[\tp])=\la\bx[\t]$ for some $\la\in\Rtnn$. Let $\t':=\Pareq(\HAT,\t)$ and $\tp':=\Pareq(\HAT,\tp)$, thus $\t'\subseteq\tp'$. Suppose first that $\t'\subsetneq\tp'$. Then $\t'$ is a subset of some \pipe in $\HAT[\tp']$, and thus $\psz^{\t'}(\bx[\tp'])=0$. Since $\bx[\tp]$ is proportional to $\psz^{\tp}(\bx[\tp'])$, and since $\psz^{\t}\circ\psz^{\tp}=\psz^{\t}\circ\psz^{\t'}=\psz^\t$ (cf. \cref{rmk:suppress}), it follows that $\psz^\t(\bx[\tp])=0$. Thus~\eqref{eq:coh} holds with $\la=0$. Suppose now that $\t'=\tp'$ and let $\by:=\bx[\t']=\bx[\tp']$. Then $\bx[\t]$ is a positive scalar multiple of $\psz^\t(\by)$ and $\bx[\tp]$ is a positive scalar multiple of $\psz^{\tp}(\by)$. Again using $\psz^{\t}\circ\psz^{\tp}=\psz^\t$, we find that $\bx[\t]$ is a positive scalar multiple of $\psz^\t(\bx[\tp])$. Thus $\bx\in\Coh(P)=\Comp(P)$. Moreover, \cref{dfn:bx_to_tubing} implies that $\bx\in\DT$.

We have constructed a bijection between $\DT$ and $\prod_{\t\in\PTubing} \OFace(\Px[\t],\HAT[\t])$. This bijection and its inverse are clearly continuous, thus the two spaces are homeomorphic.
\end{proof}
\begin{corollary}\label{cor:Comp_cell_decn}
We have a disjoint union%
\begin{equation*}%
  \Comp(P)=\bigsqcup_{\T} \DT,
\end{equation*}
where for each proper \piping $\T$, the cell $\DT$ is homeomorphic to $\R^{|P|-|\T|-2}$. 
\end{corollary}
\begin{proof}
By \cref{prop:DT_cell}, $\DT$ is homeomorphic to an open ball. By \cref{prop:Ord_faces}, its dimension is given by 
\begin{equation*}%
  \sum_{\t\in\PT} (|\HAT[\t]|-2)=  \sum_{\t\in\PT} |\HAT[\t]| -2|\T|-2=|\T|+|P|-2|\T|-2=|P|-|\T|-2.
\end{equation*}
The first and the third equalities are trivial, and the second equality follows from the fact that each \pipe $\tm\in\HAT\setminus\{P\}$ appears in $\HAT[\t]$ for exactly one $\t\in\PT$.
\end{proof}

\begin{lemma}\label{lemma:DT_closure}
The closure of each cell $\DT$ in $\Comp(P)$ is given by
\begin{equation*}%
  \Closure(\D_\Tubing)=\bigsqcup_{\Tubing'\supseteq\Tubing} \D_{\Tubing'},
\end{equation*}
where the union is taken over proper \pipings $\T'$ containing $\T$.
\end{lemma}
\begin{proof}
Suppose that a point $\bx\in\D_{\Tubing'}$ belongs to $\Closure(\D_\Tubing)$. First, we show that $\Tubing'\supseteq\Tubing$. Let $\t\in\T$ and $\tp:=\Pareq(\HAT',\t)$. We need to show that $\t\in\T'$. If $\t=\tp$ then we are done, so assume $\t\subsetneq\tp$. 

We claim that for any point $\by\in\DT$, $\t$ is contained inside some \pipe $\t'\in\tubes(\by[\tp])$ (which therefore satisfies $\t'\subsetneq\tp$). Indeed, by \cref{lemma:Coh_reconstruct}, $\by[\tp]$ is obtained as $\res_{\tp}(\by[\tp'])$ for $\tp':=\Pareq(\HAT,\tp)$. We see that $\t,\tp'\in\HAT$ and $\t\subsetneq \tp\subseteq \tp'$, so $\t$ is contained inside some \pipe in $\tubes(\by[\tp'])$. Since $\t\subsetneq\tp$, it follows that $\t$ is contained inside some $\t'\in\tubes(\by[\tp])$.

Since $\bx$ is the limit of a sequence of points in $\DT$, we see that $\t$ is contained inside some \pipe $\t''\in\tubes(\bx[\tp])$. By \cref{dfn:bx_to_tubing}, we have $\t''\in\HAT'$, and since $\t\subseteq\t''\subsetneq\tp$, we get a contradiction with the minimality of $\tp$. We have shown that $\T'\supseteq \T$.

Conversely, suppose that $\bx\in\D_{\T'}$ for some $\T'\supseteq\T$. Our goal is to show that $\bx\in\Closure(\DT)$. We modify the construction in the proof of \cref{prop:Comp=Coh}. Choose a vector $\bt=(t_\t)_{\t\in\T'\setminus\T}$. For $\t\in\PT$, define a vector $\by^\pbt[\t]$ by
\begin{equation*}%
  y_i^\pbt[\t]:=x_i[\t]+\sum_{\tm\in\T'\setminus\T:\ i\in\tm\subsetneq \t} t_{\tm}x_i[\tm] \quad\text{for $i\in\t$.}
\end{equation*}
For $\bt\in\Rtp^{\T'\setminus\T}$ sufficiently small satisfying~\eqref{eq:coh_suff_small}, we get $\by^\pbt[\t]\in\COrd(\t)\setminus\{\bnull\}$. Let
\begin{equation*}%
  \bz^\pbt[\t]:=\res_\t(\by^\pbt[\t])\quad\in\Ord(\t).
\end{equation*}
We see that $\bz^\pbt[\t]\in\OFace(\Px[\t],\HAT[\t])$. Repeating this for each $\t\in\PT$, we obtain a point in $\prod_{\t\in\PTubing} \OFace(\Px[\t],\HAT[\t])$, which by \cref{prop:DT_cell} gives a point $\bz^\pbt\in\DT$. Similarly to the argument in the proof of \cref{prop:Comp=Coh}, we get $\bz^\pbt\to\bx$ as $\bt\to\bnull$.
\end{proof}

\subsection{Collapsing and expanding maps}\label{sec:collapse_expand}
We now come to the most technical part of our proof. We will construct a family of maps which will be later used to show that the closure $\Closure(\DT)$ of each cell is a topological manifold with boundary. Throughout this section, we fix two \pipes $\t\subsetneq \tp$ with $|\t|>1$. 

\begin{definition}
Given a proper \piping $\T$, we say that $\t,\tp$ are \emph{adjacent in $\HAT$} if $\t,\tp\in\HAT$ and $\tp$ is the parent of $\t$ in $\HAT$, i.e., $\t\in\HAT[\tp]$. We denote by $\adj(\t,\tp)$ the set of proper \pipings $\T$ such that $\t,\tp$ are adjacent in $\HAT$.  We let
\begin{equation*}%
  \Adj(\t,\tp):= \bigsqcup_{\T\in\adj(\t,\tp)} \DT, \quad \Adjp(\t,\tp):= \bigsqcup_{\T\in\adj(\t,\tp)} \DT\sqcup \D_{\T\setminus\{\t\}}.
\end{equation*}
\end{definition}

Next, we write
\begin{equation}\label{eq:ijP_jiP}
  \ijP:=\{(i,j)\in\t\times(\tp\setminus\t)\mid i\lP j\} \quad\text{and}\quad \jiP:=\{(j,i)\in(\tp\setminus\t)\times\t\mid j\lP i\}.
\end{equation}

For $\bx\in\Adj(\t,\tp)$, let%
\begin{equation}\label{eq:tmax_dfn}
  \tmax_{\t,\tp}(\bx):=\sup\left\{t\in\R_{\geq0}\middle | 
\begin{tabular}{ccl}
                                                    $\coordx(\tp,i)+t\coordx(\t,i)<\coordx(\tp,j)$ &\quad & for $(i,j)\in\ijP$, and \\
 $\coordx(\tp,j)<\coordx(\tp,i)+t\coordx(\t,i)$& & for $(j,i)\in\jiP$
                                                  \end{tabular} \right\}.
\end{equation}
Note that the set on the right hand side of~\eqref{eq:tmax_dfn} is nonempty since it contains $t=0$. Thus we get a map $\tmax_{\t,\tp}:\Adj(\t,\tp)\to [0,\infty]$. We treat $[0,\infty]$ as a topological space homeomorphic to a line segment.
\begin{lemma}\label{lemma:tmax_cont}
The map $\tmax_{\t,\tp}$ is continuous on $\Adj(\t,\tp)$ and has image in $(0,\infty]$.
\end{lemma}
\begin{proof}
We will show instead that $\frac1{\tmax_{\t,\tp}}$ is a continuous function $\Adj(\t,\tp)\to [0,\infty)$. Observe that $\coordx(\tp,i)<\coordx(\tp,j)$ for all $\bx\in\Adj(\t,\tp)$ and $(i,j)\in\ijP$. Thus  $f_{i,j}(\bx):=\frac{\coordx(\t,i)}{\coordx(\tp,j)-\coordx(\tp,i)}$ is a continuous function $\Adj(\t,\tp)\to\R$, and therefore $f^+_{i,j}(\bx):=\max(f_{i,j}(\bx),0)$ is a continuous function with image in $\R_{\geq0}$. Similarly, for $(j,i)\in\jiP$, let $g_{i,j}(\bx):=\frac{-\coordx(\t,i)}{\coordx(\tp,i)-\coordx(\tp,j)}$ and $g^+_{i,j}(\bx):=\max(g_{i,j}(\bx),0)$. It follows from~\eqref{eq:tmax_dfn} that %
\begin{equation*}%
  \frac1{\tmax_{\t,\tp}(\bx)}=\min \left( \{f^+_{i,j}(\bx)\mid (i,j)\in\ijP\}\cup\{g^+_{i,j}(\bx)\mid (j,i)\in\jiP\}\right) \in[0,\infty).
\end{equation*}
In particular, $\frac1{\tmax_{\t,\tp}}$ is continuous since it is the minimum of several continuous functions.
\end{proof}

Define the \emph{expanding set}
\begin{equation*}%
  \Ex(\t,\tp):=\{(\bx,t)\in\Adj(\t,\tp)\times [0,\infty)\mid 0\leq t< \tmax_{\t,\tp}(\bx)\}.
\end{equation*}
 Similarly, define the \emph{collapsing set}
\begin{equation}\label{eq:Coll_dfn}
  \Coll(\t,\tp):=\left\{\bx\in\Adjp(\t,\tp) \middle|  \begin{tabular}{l}
                                                    $\avg_\t(\bx\tbr[\tp])<\coordx(\tp,j)$\quad for $(i,j)\in\ijP$, and \\
$\coordx(\tp,j)<\avg_\t(\bx\tbr[\tp])$\quad for $(j,i)\in\jiP$
                                                  \end{tabular} \right\}.
\end{equation}
The following result is a straightforward consequence of the definitions and \cref{lemma:tmax_cont}.
\begin{lemma}\label{lemma:Ex_Adj_open_single_ttp}
\ 
  \begin{theoremlist}
  \item $\Ex(\t,\tp)$ is an open subset of $\Adj(\t,\tp)\times [0,\infty)$ containing $\Adj(\t,\tp)\times\{0\}$.
  \item $\Coll(\t,\tp)$ is an open subset of $\Adjp(\t,\tp)$ containing $\Adj(\t,\tp)$. \qed
\end{theoremlist}
\end{lemma}

Finally, we introduce \emph{expanding and collapsing maps}. We first define the \emph{expanding map} 
\begin{equation*}%
  \ex_{\t,\tp}: \Ex(\t,\tp)\to \Coll(\t,\tp).
\end{equation*}
 Let $(\bx,t)\in\Ex(\t,\tp)$. If $t=0$, we set $\ex_{\t,\tp}(\bx,t):=\bx$. If $t>0$, the image $\ex_{\t,\tp}(\bx,t)=\by$ is described as follows. Let $\T:=\T(\bx)$, thus $\t,\tp\in\T$, and let $\T':=\T\setminus\{\t\}$. The point $\by$ will belong to $\D_{\T'}$, thus by \cref{prop:DT_cell}, it suffices to specify a point $\by[\t']\in\OFace(\Px[\t'],\HAT'[\t'])$ for each $\t'\in\PTp$. For $\t'\in\T'\setminus\{\tp\}$, set $\by\tbr[\t']:=\bx\tbr[\t']$. Let $\bz\in\COrd(\Px[\tp])\setminus\{\bnull\}$ be defined by
\begin{equation}\label{eq:ex_dfn_z}
  z_i:=
  \begin{cases}
    \coordx(\tp,i), &\text{if $i\in \tp\setminus\t$;}\\
    \coordx(\tp,i)+t\coordx(\t,i), &\text{if $i\in \t$.}\\
  \end{cases}
\end{equation}
Set $\by\tbr[\tp]:=\frac1{\al_{\Px[\tp]}(\bz)} \bz$. Thus indeed $\by[\t']\in\OFace(\Px[\t'],\HAT'[\t'])$ for each $\t'\in\PTp$, and by \cref{prop:DT_cell}, this data gives rise to a point $\by\in\D_{\T'}$. Since the conditions in the definition of $\Coll(\t,\tp)$ are satisfied for $\bz$ (where $\avg_\t(\bz)=\coordx(\tp,i)$ for any $i\in\t$), we find $\by\in\Coll(\t,\tp)$. We set $\ex_{\t,\tp}(\bx):=\by$.

Next, we describe the \emph{collapsing map} 
\begin{equation*}%
  \coll_{\t,\tp}: \Coll(\t,\tp)\to\Ex(\t,\tp).
\end{equation*}
 We will later see that it is the set-theoretic inverse of $\ex_{\t,\tp}$. Let $\by\in\Coll(\t,\tp)$ and let $\T':=\T(\by)$. If $\t\in\T'$, we set $\coll(\by):=(\by,0)$. Suppose now that $\t\notin\T'$ and set $\T:=\T'\sqcup\{\t\}$. Introduce a point $\bz\in\COrd(\Px[\tp])$ given by
\begin{equation}\label{eq:coll_y_to_z}
  z_i:=
  \begin{cases}
    \coordy(\tp,i), &\text{if $i\in\tp\setminus\t$;}\\
    \avg_\t(\by\tbr[\tp]), &\text{if $i\in\t$.}\\
  \end{cases}
\end{equation}
Set $\bx\tbr[\t']:=\by[\t']$ for all non-singleton $\t'\in\HAT\setminus\{\tp\}$ (including the case $\t'=\t$), thus $\bx\tbr[\t']\in\OFace(\Px[\t'],\HAT[\t'])$. Set $\bx\tbr[\tp]:=\frac1{\al_{\Px[\tp]}(\bz)} \bz$. 
 Applying \cref{prop:DT_cell}, we obtain a point $\bx\in\DT$. We let $t\in\R_{\geq0}$ be the unique number satisfying 
\begin{equation*}%
  \frac1{\al_{\Px[\tp]}(\bz)} \coordy(\tp,i)=\coordx(\tp,i)+t\coordx(\t,i) \quad\text{for all $i\in\t$}.
\end{equation*}
We see that $0<t<\tmax_{\t,\tp}(\bx)$ since the inequalities in the definition~\eqref{eq:tmax_dfn} of $\tmax_{\t,\tp}$ are satisfied for $\by\tbr[\tp]$. We set $\coll_{\t,\tp}(\by):=(\bx,t)\in\Ex(\t,\tp)$.

\begin{proposition}\label{prop:ex_cont}
The maps $\ex_{\t,\tp}$ and $\coll_{\t,\tp}$ are mutually inverse homeomorphisms between $\Ex(\t,\tp)$ and $\Coll(\t,\tp)$.
\end{proposition}
\begin{proof}
The fact that these maps are set-theoretic inverses of each other follows by construction. It remains to check that both maps are continuous. Let $(\bx,t)\in\Ex(\t,\tp)$. If $t>0$ then $\ex_{\t,\tp}$ is obviously continuous at $(\bx,t)$, so suppose $t=0$. Choose a sequence
 $(\bx^\parr n, t^\parr n)\in \Ex(\t,\tp)$ satisfying $0\leq t^\parr n<\tmax_{\t,\tp}(\bx^\parr n)$ and converging to $(\bx,0)$ as $n\to\infty$. Let $\by^\parr n:=\ex_{\t,\tp}(\bx^\parr n,t^\parr n)$. Since $\ex_{\t,\tp}(\bx,0)=\bx$, we need to show that $\lim_{n\to\infty} \by^\parr n=\bx$. Without loss of generality, we may assume that $\bx^\parr n\in \D_{\T'}$ for some fixed $\T'$. 

Letting $\T:=\T(\bx)$, we see that $\T'\subseteq \T$ by \cref{lemma:DT_closure}. Since $\bx,\bx^\parr n\in\Adj(\t,\tp)$, we have $\T,\T'\in\adj(\t,\tp)$, thus $\t,\tp\in\HAT'\subseteq \HAT$. Let $\t'$ be any non-singleton \pipe, and let $\tp':=\Pareq(\HAT',\t')$. We consider four cases:
\begin{enumerate}[label=\normalfont(\arabic*)]
\item\label{tp1} $\tp'\neq\t,\tp$;
\item\label{tp2} $\tp'=\t$;
\item\label{tp3} $\tp'=\tp$ and $\t'\cap\t=\emptyset$;
\item\label{tp4} $\tp'=\tp$ and $\t'\cap\t\neq\emptyset$.
\end{enumerate}
We use \cref{lemma:Coh_reconstruct} to show that in cases~\ref{tp1}--\ref{tp3}, we have $\by^\parr n[\t']=\bx^\parr n[\t']$. First, in case~\ref{tp1},
\begin{equation*}%
  \bx\pan[\t']=\res_{\t'}(\bx\pan[\tp'])=\res_{\t'}(\by\pan[\tp'])=\by\pan[\t'].
\end{equation*}
In case~\ref{tp2}, by~\eqref{eq:ex_dfn_z}, we have $\bx\pan[\t]=\res_\t(\by\pan[\tp])$, and thus
\begin{equation*}%
  \bx\pan[\t']=\res_{\t'}(\bx\pan[\t])=\res_{\t'}(\res_\t(\by\pan[\tp]))=\res_{\t'}(\by\pan[\tp])=\by\pan[\t'].
\end{equation*}
In case~\ref{tp3}, since $\res_{\t'}(\bz)$ depends only on the coordinates $z_i$ for $i\in\t'$, we get
\begin{equation*}%
  \bx\pan[\t']=\res_{\t'}(\bx\pan[\tp])=\res_{\t'}(\by\pan[\tp])=\by\pan[\t'].
\end{equation*}
Therefore in cases~\ref{tp1}--\ref{tp3}, we find
\begin{equation*}%
  \lim_{n\to\infty} \by^\parr n[\t']=  \lim_{n\to\infty} \bx^\parr n[\t']= \bx\tbr[\t'], \quad\text{since} \lim_{n\to\infty} \bx^\parr n=\bx.
\end{equation*}
In case~\ref{tp4}, because $\t,\tp$ are adjacent in $\HAT$, we get $\Pareq(\HAT,\t')=\tp$. Thus $\by\pan[\t']=\res_{\t'}(\by\pan[\tp])$ and $\bx[\t']=\res_{\t'}(\bx[\tp])$. By construction~\eqref{eq:ex_dfn_z}, we have $\by\pan[\tp]\to\bx[\tp]$ as $n\to\infty$, which implies the result by the continuity of $\res_{\t'}$. We have shown that the map $\ex_{\t,\tp}$ is continuous.

We now check the continuity of $\coll_{\t,\tp}$. Let $\by\in\Coll(\t,\tp)$ with $\T:=\T(\by)$. If $\t\notin\T$ then clearly $\coll_{\t,\tp}$ is continuous at $\by$, so assume $\t\in\T$. Thus $\coll_{\t,\tp}(\by)=(\by,0)$. Choose a sequence $\by\pan$ in $\Coll(\t,\tp)$ converging to $\by$, and assume that $\T'':=\T(\by^\parr n)$ is fixed. By \cref{lemma:DT_closure}, it satisfies $\T''\subseteq\T$. It $\t\in\T''$ then $\coll_{\t,\tp}(\by\pan)=(\by\pan,0)$ converges to $\by$ as $n\to\infty$, so assume $\t\notin\T''$, and let $\T':=\T''\sqcup\{\t\}$. We again have $\T,\T'\in\adj(\t,\tp)$.

Let $(\bx^\parr n,t^\parr n):=\coll_{\t,\tp}(\by^\parr n)$, thus $t\pan>0$ and $\T(\bx\pan)=\T'$. Let $\t'$ be any non-singleton \pipe, and let $\tp':=\Pareq(\HAT',\t')$. Considering cases~\ref{tp1}--\ref{tp4} above, we check that  
\begin{equation}\label{eq:lim_bxpan_by}
  \lim_{n\to\infty}\bx\pan=\by.
\end{equation}

It remains to show that $t^\parr n\to 0$ as $n\to\infty$. Let $\bz$ and $\bz^\parr n$ be obtained respectively from $\by$ and $\by^\parr n$ via~\eqref{eq:coll_y_to_z}. Thus $\lim_{n\to\infty}\bz^\parr n= \bz$. Choose $i\lP j$ with $i,j\in\t$. For each $n$, $t\pan$ satisfies
\begin{equation*}%
    \frac1{\al_{\Px[\tp]}(\bz^\parr n)} (\coordypn(\tp,j)-\coordypn(\tp,i))=(\coordxpn(\tp,j)-\coordxpn(\tp,i))+t^\parr n(\coordxpn(\t,j)-\coordxpn(\t,i)).
\end{equation*}
By~\eqref{eq:lim_bxpan_by}, $\coordxpn(\t,j)-\coordxpn(\t,i)$ has a positive limit $\coordx(\t,j)-\coordx(\t,i)$, and since we have
\begin{equation*}%
  \frac1{\al_{\Px[\tp]}(\bz)} (\coordy(\tp,j)-\coordy(\tp,i))=(\coordx(\tp,j)-\coordx(\tp,i))+t(\coordx(\t,j)-\coordx(\t,i)) \quad\text{for $t=0$,}
\end{equation*}
we find $t^\parr n\to 0$ as $n\to\infty$.
\end{proof}

\subsection{$\Comp(P)$ is a topological manifold with boundary}\label{sec:Comp_top_manif}
Our next goal is to show that $\Comp(P)$ --- as well as each cell closure $\Closure(\D_\Tubing)$ --- is a topological manifold with boundary. 

Fix two proper \pipings $\T'\subseteq\T$. Let
\begin{equation}\label{eq:T_T'_by_inclusion}
  \Tubing\setminus\Tubing'=\{\tube^\parr1,\dots,\tube^\parr m\}
\end{equation}
be ordered by inclusion so that $\t^\parr i \subseteq\t^\parr j$ implies $i\leq j$. 
For $i\in[m]$, let $\tp\pai\supsetneq\t\pai$ be the parent of $\t\pai$ in $\HAT$ (cf. \cref{dfn:tree}). 

Let $(\bx_0,\bt)\in \DT\times[0,\infty)^m$. We give the following inductive definition. We say that $(\bx_0,\bt)$ is \emph{$0$-expandable} and set $\ex_{\T,\T'}^\parr0(\bx_0,\bt):=\bx_0$. For each $i=1,2,\dots,m$, assume that $(\bx_0,\bt)$ is $(i-1)$-expandable and set $\bx_{i-1}:=\ex_{\T,\T'}^\parr{i-1}(\bx_0,\bt)$. We say that $(\bx_0,\bt)$ is \emph{$i$-expandable} if $(\bx_{i-1},t_i)\in\Ex(\ti,\tpi)$. In this case, we define $\ex_{\T,\T'}\pai(\bx_0,\bt):=\ex_{\ti,\tpi}(\bx_{i-1},t_i)$. 
 Let
\begin{equation*}%
  \Ex(\T,\T'):=\{(\bx_0,\bt)\in \DT\times[0,\infty)^m\mid (\bx_0,\bt)\text{ is $m$-expandable}\}.
\end{equation*}
We thus get a map $\ex_{\T,\T'}:=\ex_{\T,\T'}^\parr m:\Ex(\T,\T')\to\Comp(P)$.

 Conversely, set
\begin{equation*}%
  \Star(\Tubing,\Tubing'):=\bigsqcup_{\Tubing'\subseteq \Tubing''\subseteq \Tubing} \D_{\Tubing''}.%
\end{equation*}
Clearly, the image of $\ex_{\T,\T'}$ is contained in $\Star(\T,\T')$. Let $\by_m\in \Star(\T,\T')$. We say that $\by_m$ is \emph{$m$-collapsible} and define $\coll^\parr m(\by_m):=\by_m$. For each $i=m,m-1,\dots,1$, assume that $\by_m$ is $i$-collapsible and that we have defined a point $\coll\pai(\by_m)=(\by_i,\bt\pai)$, where $\bt\pai=(t_m,t_{m-1},\dots,t_{i+1})\in[0,\infty)^{m-i}$. We say that $\by_m$ is \emph{$(i-1)$-collapsible} if $\by_i\in\Coll(\ti,\tpi)$. In this case, denoting $(\by_{i-1},t_i):=\coll_{\ti,\tpi}(\by_{i})$, we define $\coll^\parr{i-1}(\by_m):=(\by_{i-1},(t_m,t_{m-1},\dots,t_{i+1},t_{i}))$. 
 Let
\begin{equation*}%
  \Coll(\T,\T'):=\{\by_m\in \Star(\T,\T')\mid \by_m\text{ is $0$-collapsible}\}.
\end{equation*}
We thus have a map $\coll_{\T,\T'}:=\coll_{\T,\T'}^\parr0:\Coll(\T,\T')\to \DT\times[0,\infty)^m$.
By \cref{prop:ex_cont}, $\ex_{\T,\T'}$ and $\coll_{\T,\T'}$ form a pair of mutually inverse homeomorphisms between $\Ex(\T,\T')$ and $\Coll(\T,\T')$.

\begin{lemma}\ 
  \begin{theoremlist}
  \item\label{item:Ex_open} $\Ex(\T,\T')$ is an open subset of $\DT\times[0,\infty)^m$ containing $\DT\times\{\bnull\}$.
  \item\label{item:Coll_open} $\Coll(\T,\T')$ is an open subset of $\Closure(\D_{\T'})$ containing $\DT$.
  \end{theoremlist}
\end{lemma}
\begin{proof}
\itemref{item:Ex_open}: For $i=0,1,\dots,m$, let $\Ex\pai(\T,\T')\subset\DT\times[0,\infty)^m$ be the set of $i$-expandable points. Thus $\Ex^\parr0(\T,\T')=\DT\times[0,\infty)^m$ and $\Ex^\parr m(\T,\T')=\Ex(\T,\T')$. 
 We proceed by induction on $i=1,2,\dots,m$. Suppose that $\Ex^\parr{i-1}(\T,\T')$ is open inside $\DT\times[0,\infty)^m$ and contains $\DT\times\{\bnull\}$. We have 
\begin{equation*}%
  \Ex\pai(\T,\T')=\{(\bx_0,\bt)\in \Ex^\parr{i-1}(\T,\T')\mid (\ex_{\T,\T'}^\parr{i-1}(\bx_0,\bt),t_i)\in \Ex(\ti,\tpi)\}.
\end{equation*}
Observe that for any $(\bx_0,\bt)\in \Ex^\parr{i-1}(\T,\T')$, the point $\bx_{i-1}:=\ex_{\T,\T'}^\parr{i-1}(\bx_0,\bt)$ belongs to $\Adj(\ti,\tpi)$ since the \pipes in~\eqref{eq:T_T'_by_inclusion} are ordered by inclusion. In order for $(\bx_{i-1},t_i)$ to belong to $\Ex(\ti,\tpi)$, we must have $t^\parr i<\tmax_{\ti,\tpi}(\bx_{i-1})$. By \cref{lemma:Ex_Adj_open_single_ttp}, $\Ex(\ti,\tpi)$ is open in $\Adj(\ti,\tpi)$. Since the maps $\ex_{\T,\T'}^\parr{i-1}$ and $\tmax_{\ti,\tpi}$ are continuous, it follows that $\Ex\pai(\T,\T')$ is open in $\DT\times[0,\infty)^m$. By construction, it contains $\DT\times\{\bnull\}$. This finishes the induction step.

\itemref{item:Coll_open}: Similarly, for $i=m,m-1,\dots,0$, let $\Coll^\parr i(\T,\T')\subseteq \Star(\T,\T')$ consist of all $i$-collapsible points $\by_m$.
 Denote $(\by_i,\bt\pai):=\coll\pai(\by_m)$ as above. It follows that $\by_i\in\Adjp(\ti,\tpi)$ for each $i$. By \cref{lemma:Ex_Adj_open_single_ttp}, $\Coll(\ti,\tpi)$ is an open subset of $\Adjp(\ti,\tpi)$. Thus each $\Coll^\parr i(\T,\T')$ is an open subset of $\Star(\T,\T')$, which is an open subset of $\Closure(\D_{\T'})$. By construction, $\Coll^\parr i(\T,\T')$ contains $\DT$ for each $i=m,m-1,\dots,0$.
\end{proof}

\begin{corollary}
Each cell closure
\begin{equation}\label{eq:Closure_is_top_mnfld}
  \Closure(\D_{\T'})=\bigsqcup_{\Tubing\supseteq\Tubing'} \D_{\T}
\end{equation}
 is a topological manifold with boundary
\begin{equation}\label{eq:Closure_is_top_mnfld_bdry}
  \partial\D_{\T'}=\bigsqcup_{\Tubing\supsetneq\Tubing'} \D_{\T}.
\end{equation}
\end{corollary}
\noindent Note that $\Comp(P)=\Closure(\D_{\emptyset})$ appears in the above corollary as a special case.
\begin{proof}
Choose a point $\by\in \Closure(\D_{\T'})$ and let $\T=\T(\by)$. We have constructed a homeomorphism $\Ex(\T,\T')\xrasim\Coll(\T,\T')$. We have $\by\in\Coll(\T,\T')$ and $(\by,\bnull)\in\Ex(\T,\T')$. Since $\Ex(\T,\T')$ is open, we can choose an open neighborhood  $U\times[0,\eps)^m\subset \Ex(\T,\T')$ of $(\by,\bnull)$ such that $U$ is homeomorphic to an open ball. Then the image of $U\times[0,\eps)^m$ under $\ex_{\T,\T'}$ is an open neighborhood of $\by$ inside $\Coll(\T,\T')$, which is open inside $\Closure(\D_{\T'})$. 
Thus $\by$ admits an open neighborhood inside $\Closure(\D_{\T'})$, homeomorphic to $U\times[0,\eps)^m$, where $m=|\T\setminus\T'|$. If $m>0$ then $U\times[0,\eps)^m$ is homeomorphic to a half-space, and if $m=0$ then $U$ is homeomorphic to an open ball.
\end{proof}

\begin{proof}[Proof of Theorem~\ref{thm:intro:comp}]
  Since $\Comp(P)$ is a subset of $\prod_{|\t|>1} \Ord(\Px[\t])$, it is Hausdorff. We have constructed a stratification of $\Comp(P)$ into cells so that the closure of each cell is a topological manifold with boundary, and the boundary of each cell is the union of lower cells. Moreover, the poset of cell closures (i.e., the poset of proper \pipings ordered by reverse inclusion) is isomorphic to the face poset of the polytope $\Ass(P)$. It is then a standard application of the \emph{generalized \Poincare conjecture}~\cite{Sma,Fre,Per1,Per2,Per3} that $\Comp(P)$ is a \emph{regular CW complex}  homeomorphic to $\Ass(P)$. We outline a proof sketch and refer to~\cite[Section~3.2]{GKL3} for full details.

The proof proceeds by induction on cell dimension. Given a cell $\DT$, by the induction hypothesis, the closure of each cell in $\partial\DT$ is homeomorphic to a closed ball, with boundary homeomorphic to a sphere. This endows $\partial\DT$ with the structure of a regular CW complex. Its cell closure poset is isomorphic to the face poset of the boundary of the face of $\Ass(P)$ labeled by $\T$. Thus it follows from the results of~\cite{Bjo} that $\partial\DT$ is homeomorphic to a sphere. Since $\Closure(\DT)$ is a topological manifold with boundary, by an application of the generalized \Poincare conjecture (see~\cite[Theorem~10.3.3(ii)]{Davis} or~\cite[Theorem~3.10]{GKL3}), $\Closure(\DT)$ is homeomorphic to a closed ball. This constitutes the induction step. 
\end{proof}

\section{Affine \aposetdash cyclohedra}\label{sec:affine_assoc}
Let $\Pa$ be an affine poset of order  $|\Pa|=n\geq1$. We explain how our results on \posetdash associahedra extend to affine \aposetdash cyclohedra. For the most part, the proofs are completely analogous; we indicate the places where they differ from their \posetdash associahedra counterparts. Throughout this section, by \emph{\pipes} and \emph{\pipings} we mean $\Pa$-\pipes and $\Pa$-\pipings, respectively.

For our purposes, it will be more convenient to slightly change the definition~\eqref{eq:intro:Orda} of $\Orda(\Pa)$ and $\Oao(\Pa)$ and work with $\RSZ^{|\Pa|}$ rather than with $\R^{|\Pa|}/\R(1,1,\dots,1)$:
\begin{equation*}%
  \Oao(\Pa):=\{\bx\in\RSZ^{|\Pa|}\mid \xt_i<\xt_j\text{ for all $i\lpa j$}\},\quad \Orda(\Pa):=\{\bx\in\RSZ^{|\Pa|}\mid \xt_i\leq\xt_j\text{ for all $i\leqpa j$}\}.
\end{equation*}
Our first goal is to show that $\Orda(\Pa)$ is nonempty.
\begin{definition}
A \emph{linear extension} of $\Pa$ is a bijection $\phi:\Z\to\Z$ satisfying $\phi(i+n)=\phi(i)+n$ and $\phi(i)<\phi(j)$ for all $i\lPa j$.
\end{definition}
\noindent For instance, the vertex labels of the affine posets shown in Figures~\ref{fig:aff_cyc} and~\ref{fig:crit} are examples of linear extensions.
\begin{lemma}
Each affine poset $\Pa$ admits at least one linear extension.
\end{lemma}
\begin{proof}
Let $S:=\{i\in\Z\mid i-n\lPa 0\text{ and }i\not\lPa 0\}$. Because $\Pa$ is strongly connected (cf. \cref{dfn:intro:aff_poset}), $S$ contains exactly one element in each residue class modulo $n$, thus $|S|=n$. Moreover, we claim that $S$ is convex. Indeed, suppose $i,j,k\in\Z$ are such that $i\lpa j\lpa k$ and $i,k\in S$. Then we have $j-n\lpa k-n\lpa 0\not\lpa i\lpa j$, so $j\in S$. Note, however, that $S$ need not be connected in general.

Consider $S$ as a finite subposet $(S,\leqpa)$ of $\Pa$. Choose a linear extension $\bar\phi:\Pax[S]\to[n]$ of $\Pax[S]$, and let $\phi:\Z\to\Z$ be its unique $n$-periodic extension (satisfying $\phi(i)=\bar \phi(i)$ for $i\in S$ and $\phi(i+n)=\phi(i)+n$ for $i\in\Z$). We claim that $\phi$ is a linear extension of $\Pa$. First, it is clearly a bijection $\Z\to\Z$. Second, suppose that for some $i\lPa j$, we have $\phi(i)>\phi(j)$. Adding a multiple of $n$ to both indices, we may assume that $j\in S$. Let $i'\in S$ be such that $i'\equiv i\pmod n$. Since $\bar\phi$ is a linear extension, we cannot have $i=i'$. If $i<i'$ then because $\phi(i'),\phi(j)\in[n]$, we have $\phi(i)\leq \phi(i')-n\leq 0<\phi(j)$, a contradiction. Assume now that $i'<i$. Then $i'\lpa i\lpa j$, so since $S$ is convex, we get $i\in S$, a contradiction. 
\end{proof}

\begin{corollary}\label{cor:Orda_nonempty}
$\Orda(\Pa)$ is a nonempty polytope of dimension $n-1$.
\end{corollary}
\begin{proof}
Let $\phi$ be a linear extension of $\Pa$. Setting $x_i:=\phi(i)\cdot \frac cn$ for $i\in[n]$, we obtain a point $\bx\in\R^{|\Pa|}$ such that $\psz^{[n]}(\bx)\in\Oao(\Pa)$. Thus the interior of $\Orda(\Pa)$ in $\RSZ^{|\Pa|}$ is nonempty.
\end{proof}

\begin{figure}
\includegraphics{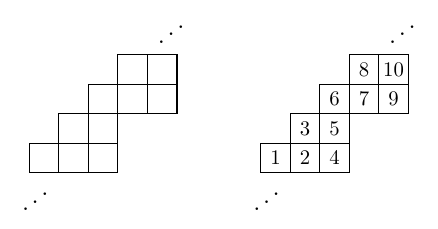}
  \caption{\label{fig:cyl} A cylindric skew shape $\Pa$ (left) and a linear extension of $\Pa$ (right).}
\end{figure}

\begin{remark}
We mention several relations between affine posets and existing objects in the literature. First, $\Orda(\Pa)$ is an \emph{alcoved polytope} in the sense of~\cite{LaPo_alcoved}. Its volume is the number of different linear extensions of $\Pa$, where two linear extensions are considered the same if their values differ by a constant. It is an interesting problem to compute the number of such linear extensions for various classes of posets, such as the ones arising from critical varieties as discussed in \cref{rmk:crit}. Second, it would be interesting to develop an analogous theory of (combinatorial, piecewise-linear, or birational) \emph{rowmotion} on affine posets; see e.g.~\cite{EiPr1,EiPr2,StWi}. Third, a natural class of affine posets consists of \emph{cylindric skew shapes}, i.e., regions of $\Z^2$ between two up-right lattice paths which are invariant under shifting by some nonzero vector $(a,b)\in\Ztnn^2$. An example for $(a,b)=(2,2)$ is shown in \cref{fig:cyl}. Linear extensions of such affine posets are certain kinds of ``cylindric standard Young tableaux.'' These objects are different from the well-studied cylindric tableaux arising in quantum Schubert calculus; cf.~\cite{Pos_affine}. Indeed, the labels of the former increase in the northeast direction while the labels of the latter increase in the southeast direction.
\end{remark}

\begin{remark}\label{rmk:Pasha}
Recall from \cref{rmk:crit} that one can associate an affine poset to each permutation $f\in S_n$. A different construction associating an affine poset to an affine permutation was given in~\cite[Section~3.1]{CPY}. The authors of~\cite{CPY} consider the notion of a \emph{proper numbering} of an affine poset $\Pa$, which is a map $d:\Pa\to\Z$ such that we have $d(i)<d(j)$ for all $i\lpa j$, and such that for each $j\in\Pa$, there exists $i\lpa j$ satisfying $d(i)=d(j)-1$. Thus the notion of a proper numbering of $\Pa$ is similar but different from our notion of a linear extension of $\Pa$. It would be interesting to see which of the remarkable properties of proper numberings developed in~\cite[Section~11]{CPY} generalize to arbitrary affine posets.
\end{remark}

\begin{definition}
A \emph{\piping partition} of $\Pa$ is a \piping which is simultaneously a set partition of $\Z$.
\end{definition}
This includes the case $\T=\{\Pa\}$ which will correspond to the empty face of $\Orda(\Pa)$. Recall the notion of equivalence of \pipes from \cref{sec:intro:affine}. For a \piping $\T$ not containing $\Pa$, we let $\EQ[\T]:=\{\eq[\t]\mid \t\in\T\}$ denote the corresponding (finite) set of equivalence classes. Thus we have $\T=\bigsqcup_{\eq[\t]\in\EQ[\T]} \eq[\t]$. For the \piping partition $\T=\{\Pa\}$, we set $\EQ[\T]:=\emptyset$.

Similarly to \cref{prop:Ord_faces}, we have the following description of the faces of $\Orda(\Pa)$.
\begin{proposition}
We have a bijection $\T\mapsto \OaFacecl(\Pa,\T)$ between \piping partitions of $\Pa$ and faces of $\Orda(\Pa)$. The face closure relations are given by refinement~\eqref{eq:Ord_faces}. The dimension of each face $\OaFacecl(\Pa,\T)$ equals $|\EQ[\T]|-1$. \qed
\end{proposition}
\noindent As in the case of order polytopes, for a point $\bx\in\OaFace(\Pa,\T)$, we write $\tubes(\bx):=\T$, where $\OaFace(\Pa,\T)$ denotes the relative interior of the face $\OaFacecl(\Pa,\T)$.

We say that a \emph{maximal proper \pipe} is a \pipe $\t\neq\Pa$ satisfying $|\t|=n$. 
\begin{corollary}\label{cor:Orda_properties}\ 
\begin{enumerate}[label=\normalfont(\roman*)]
\item The vertices of $\Orda(\Pa)$ are in bijection with equivalence classes of maximal proper \pipes.
\item The facets of $\Orda(\Pa)$ are in bijection with covering relations $i\precdot_{\Pa} j$ in $\Pa$ such that $i\not\equiv j$ modulo $n$.
\item Each face $\OFacecl(\Pa,\T)$ of $\Orda(\Pa)$ is itself an affine order polytope $\Orda(\Pa/\T)$, where the quotient affine poset $\Pa/\T$ is obtained from $\Pa$ by identifying all elements that belong to a single \pipe of $\T$.
\end{enumerate}
\end{corollary}
\noindent A non-trivial consequence of \cref{cor:Orda_nonempty,cor:Orda_properties} is that the set of maximal proper \pipes is nonempty for any affine poset $\Pa$.

\begin{proof}[Proof sketch of \cref{thm:aff_poset_cyc}]
Our argument closely follows the proof of \cref{thm:poset_ass} in \cref{sec:proof:poset_ass}. We work with $n$-periodic sets $\Mcal$ of melted \pipes, still assuming that $\t\subseteq\t'$ with $\t\in\Mcal$ implies $\t'\in\Mcal$. An \emph{$\Mcal$-admissible \piping} is a \piping $\T$ containing $\Pa$ and satisfying conditions~\ref{M_adm_fro}--\ref{M_adm_melt} in \cref{sec:proof:poset_ass}. The poset $(\Adm(P;\Melt),\leq_\Melt)$ is defined in exactly the same way, using conditions~\eqref{tube_rel:frozen}--\eqref{tube_rel:melted} in \cref{sec:proof:poset_ass}. The dual affine \aposetdash cyclohedron $\Cyc(\Pa)^\ast$ is then obtained from the dual affine order polytope $\Orda(\Pa)^\ast$ via a sequence of stellar subdivisions at the faces of $\Orda(\Pa)$ corresponding to \piping partitions of the form
\begin{equation*}%
  \eq[\t]\sqcup\{\{i\}\mid i\in\Z:\text{ $i$ is not contained inside any element of $\eq[\t]$}\},
\end{equation*}
where at each step, we let $\t$ be a maximal by inclusion proper \pipe not contained in $\Melt$. 
\end{proof}

\begin{remark}
Suppose $P$ is a bounded (finite) poset with vertex set $\{0,1,\dots,n\}$ such that $0$ is the minimal element and $n$ is the maximal element. Then $P$ naturally gives rise to an affine poset $\Pa$ of order $n$ obtained by ``identifying $0$ and $n$.'' More precisely, $\leqpa$ is obtained by taking the transitive closure of relations $i+dn\leqpa j+dn$ for all $d\in\Z$ and $i\leqp j$. A very similar operation was recently considered in~\cite[Remark~2.7]{R_systems}.

It is easy to see that $\Orda(\Pa)$ is linearly equivalent to Stanley's order polytope $\hat\Ord(P)$, thus the polytopes $\Ord(P)$ and $\Ord(\Pa)$ are projectively equivalent by \cref{rmk:Stanley}. Each \Pdashpipe is also a $\Pa$-\pipe. However, not all (equivalence classes of) $\Pa$-\pipes are obtained in this way, since we have $\Pa$-\pipes of the form $\t\cup (\t'+n)$ where $\t,\t'$ are disjoint proper \Pdashpipes such that $n\in\t$ and $0\in\t'$. Thus the polytopes $\Ass(P)$ and $\Cyc(\Pa)$ are not directly related to each other. For instance, if $P$ is a chain on $4$ elements then $\Ass(P)$ is a pentagon and $\Cyc(\Pa)$ is a hexagon; compare \figref{fig:ass_perm}(left) and  \figref{fig:aff_cyc}(left).
\end{remark}

Next, we state an affine analog of \cref{cor:poset_ass_properties}, where we again identify two polytopes if they are combinatorially equivalent. We say that $\Pa$ is a \emph{circular chain} if $\leqpa$ coincides with the standard order $\leq$ on $\Z$. We say that $\Pa$ is a \emph{circular claw} if $\leqpa$ is the $n$-periodic transitive closure of relations $0\lpa 1,2,\dots,n-1\lpa n$. See \cref{fig:aff_cyc}.
\begin{corollary}\label{cor:aff_poset_cyc_properties}
Let $\Pa$ be an affine poset.
\begin{enumerate}[label=\normalfont(\roman*)]
\item\label{aPApr:simple} $\Cyc(\Pa)$ is a simple polytope of dimension $|\Pa|-1$.
\item\label{aPApr:flag} Its polar dual $\Cyc(\Pa)^\ast$ is simplicial, but in general not flag.
\item\label{aPApr:faces_dim} For each proper \piping $\T$, the corresponding face of $\Cyc(\Pa)$ has dimension $|\Pa|-|\EQ[\T]|-1$.
\item\label{aPApr:vert} The vertices of $\Cyc(\Pa)$ are in bijection with proper \pipings $\T$ satisfying $|\EQ[\T]|=|\Pa|-1$.
\item\label{aPApr:facets} The facets of $\Cyc(\Pa)$ are in bijection with equivalence classes of proper \pipes.
\item\label{aPApr:faces} Each face of $\Cyc(\Pa)$ is a product of \posetdash associahedra and affine \aposetdash cyclohedra.
\item\label{aPApr:chain} When $\Pa$ is a circular chain, $\Cyc(\Pa)$ is the $(|\Pa|-1)$-dimensional cyclohedron.
\item\label{aPApr:claw} When $\Pa$ is a circular claw, $\Cyc(\Pa)$ is the $(|\Pa|-1)$-dimensional type $B$ permutohedron.
\end{enumerate}
\end{corollary}
\begin{proof}
Each of the properties~\ref{aPApr:simple}--\ref{aPApr:claw} is either trivial or is proven similarly to its analog in \cref{cor:poset_ass_properties}. For the last two properties, we need to explain the combinatorial objects labeling the faces of the cyclohedron and the type $B$ permutohedron in order to connect them to \pipings.

\ref{aPApr:chain}: Similarly to the case of the associahedron, the faces of the $(n-1)$-dimensional cyclohedron are in bijection with rooted trees $T$ embedded in a disk such that the root has degree $\geq1$, all non-leaf vertices lie in the interior of the disk, and the leaves lie on the boundary and are labeled $0,1,2,\dots,n$ in clockwise order. Face closure relations again correspond to contracting non-leaf edges in $T$. Let $v$ be a non-leaf non-root vertex of $T$. The edges incident to $v$ have a natural cyclic order. Let $e$ be the edge connecting $v$ to its parent in $T$. Consider a walk that starts at the parent of $v$, traverses $e$ and then turns maximally left at each vertex until it reaches some leaf $a\in[n]$. Similarly, consider another walk which turns maximally right at each vertex until it reaches another leaf $b\in[n]$. The leaf descendants of $v$ naturally form a cyclic subinterval $[a,b]$ of $[n]$. We may thus associate a \pipe $\t_v$ to $v$ which equals $[a,b]$ if $a\leq b$ and $[a,b+n]$ if $a>b$. If the root of $T$ has degree $1$ and $v$ is its sole child vertex then $b$ equals $a-1$ modulo $n$, so we get $[a,b]=[n]$. Still, depending on the value of $a$, we get different \pipes $\t=[a,a+n-1]$, which corresponds to the different ways of placing the root of $T$ next to $v$. It is again clear that when $\Pa$ is a circular chain, the set of \pipes $\t_v$ where $v$ runs over the set of non-leaf non-root vertices of $T$ yields a proper $\Pa$-\piping. Moreover, we see that any proper $\Pa$-\piping arises uniquely in this way, and that contracting edges in trees corresponds to removing \pipes from a \piping.

\ref{aPApr:claw}: The $(n-1)$-dimensional \emph{type $B$ permutohedron} $\PermB_{n-1}$ is defined as the convex hull of all vectors obtained from $(1,2,\dots,n-1)$ by permuting the coordinates and changing their signs. The face poset of $\PermB_{n-1}$ coincides with the order poset of the boundary face poset of the $(n-1)$-dimensional cross-polytope; see~\cite[Example~1.3.2]{Wachs}. Thus the facets of $\PermB_{n-1}$ are in bijection with pairs $(K^+,K^-)$ of disjoint subsets of $[n-1]$ whose union is nonempty. Arbitrary faces of $\PermB_{n-1}$ are labeled by sets
\begin{equation*}%
  \{(K_1^+,K_1^-),(K_2^+,K_2^-),\dots,(K_r^+,K_r^-)\}
\end{equation*}
of such pairs satisfying the conditions
\begin{equation*}%
  K_1^+\subset K_2^+\subset\cdots\subset K_r^+\subset [n-1]\setminus K_r^-\subset [n-1]\setminus K_{r-1}^-\subset\cdots\subset [n-1]\setminus K_1^-;
\end{equation*}
see~\cite[Corollary~1.11]{Hetyei}. Identifying each pair $(K^+,K^-)$ with (the equivalence class of) the \pipe $(K^--n)\sqcup\{0\}\sqcup K^+$, we obtain a bijection between faces of $\PermB_{n-1}$ and proper $\Pa$-\pipings.
\end{proof}

It remains to justify the relation between affine \aposetdash cyclohedra and compactifications.%
\begin{proof}[Proof sketch of \cref{thm:aff_poset_cyc}]
The proof is obtained from that in \cref{sec:poset_compact} via straightforward modifications as we outline below. By convention, we write $\t\subsetneq\Pa$ for any \pipe $\t\neq \Pa$, including the case of maximal proper \pipes $\t\subsetneq\Pa$ satisfying $|\t|=|\Pa|$. Throughout the whole proof in \cref{sec:poset_compact}, we replace $P$ with $\Pa$ and $\prod_{|\t|>1}\Ord(\t)$ with $\prodb_{|\t|>1}\Orda(\t)$. The remaining changes are listed below.

By \cref{dfn:bx_to_tubing} (with $P$ replaced by $\Pa$), $\HAT(\bx)$ contains $\Pa$ and all \pipes in $\tubes(\bx[\Pa])$, thus $\HAT(\bx)$ is an infinite \piping. It still has the structure of an infinite $n$-periodic rooted tree described in \cref{dfn:tree}. The remaining definitions and proofs in \cref{sec:coh,sec:cell_decn} do not require any changes. The same applies to all results in \cref{sec:collapse_expand}, except that in \cref{cor:Comp_cell_decn}, $\Comp_{\T}(\Pa)$ is now homeomorphic to $\R^{|\Pa|-|\EQ[\T]|-1}$. The sets $\ijP$ and $\jiP$ in~\eqref{eq:ijP_jiP} become infinite when $\tp=\Pa$, but that does not affect the proof since only finitely many of their elements participate non-trivially in~\eqref{eq:tmax_dfn} and~\eqref{eq:Coll_dfn}. The rest of the proof in \cref{sec:collapse_expand,sec:Comp_top_manif} proceeds without change.
\end{proof}

\bibliographystyle{alpha} 
\bibliography{crit_polyt}

\begin{thebibliography}{PPPP19}

\bibitem[AB20]{AdBe}
Karim Adiprasito and Bruno Benedetti.
\newblock Barycentric subdivisions of convex complexes are collapsible.
\newblock {\em Discrete Comput. Geom.}, 64(3):608--626, 2020.

\bibitem[AS94]{AxSi}
Scott Axelrod and I.~M. Singer.
\newblock Chern-{S}imons perturbation theory. {II}.
\newblock {\em J. Differential Geom.}, 39(1):173--213, 1994.

\bibitem[Bj{\"o}84]{Bjo}
A.~Bj{\"o}rner.
\newblock Posets, regular {CW} complexes and {B}ruhat order.
\newblock {\em European J. Combin.}, 5(1):7--16, 1984.

\bibitem[BT94]{BoTa}
Raoul Bott and Clifford Taubes.
\newblock On the self-linking of knots.
\newblock {\em J. Math. Phys.}, 35(10):5247--5287, 1994.

\bibitem[CD06]{CaDe}
Michael~P. Carr and Satyan~L. Devadoss.
\newblock Coxeter complexes and graph-associahedra.
\newblock {\em Topology Appl.}, 153(12):2155--2168, 2006.

\bibitem[CPY18]{CPY}
Michael Chmutov, Pavlo Pylyavskyy, and Elena Yudovina.
\newblock Matrix-ball construction of affine {R}obinson-{S}chensted
  correspondence.
\newblock {\em Selecta Math. (N.S.)}, 24(2):667--750, 2018.

\bibitem[Dav08]{Davis}
Michael~W. Davis.
\newblock {\em The geometry and topology of {C}oxeter groups}, volume~32 of
  {\em London Mathematical Society Monographs Series}.
\newblock Princeton University Press, Princeton, NJ, 2008.

\bibitem[DCP95]{DCP}
C.~De~Concini and C.~Procesi.
\newblock Wonderful models of subspace arrangements.
\newblock {\em Selecta Math. (N.S.)}, 1(3):459--494, 1995.

\bibitem[DFRS15]{DFRS}
Satyan~L. Devadoss, Stefan Forcey, Stephen Reisdorf, and Patrick Showers.
\newblock Convex polytopes from nested posets.
\newblock {\em European J. Combin.}, 43:229--248, 2015.

\bibitem[EP14]{EiPr2}
David Einstein and James Propp.
\newblock Piecewise-linear and birational toggling.
\newblock In {\em 26th {I}nternational {C}onference on {F}ormal {P}ower
  {S}eries and {A}lgebraic {C}ombinatorics ({FPSAC} 2014)}, Discrete Math.
  Theor. Comput. Sci. Proc., AT, pages 513--524. Assoc. Discrete Math. Theor.
  Comput. Sci., Nancy, 2014.

\bibitem[EP21]{EiPr1}
David Einstein and James Propp.
\newblock Combinatorial, piecewise-linear, and birational homomesy for products
  of two chains.
\newblock {\em Algebr. Comb.}, 4(2):201--224, 2021.

\bibitem[FM94]{FuMa}
William Fulton and Robert MacPherson.
\newblock A compactification of configuration spaces.
\newblock {\em Ann. of Math. (2)}, 139(1):183--225, 1994.

\bibitem[Fre82]{Fre}
Michael~Hartley Freedman.
\newblock The topology of four-dimensional manifolds.
\newblock {\em J. Differential Geom.}, 17(3):357--453, 1982.

\bibitem[FS05]{FeSt}
Eva~Maria Feichtner and Bernd Sturmfels.
\newblock Matroid polytopes, nested sets and {B}ergman fans.
\newblock {\em Port. Math. (N.S.)}, 62(4):437--468, 2005.

\bibitem[FZ02]{FZ}
Sergey Fomin and Andrei Zelevinsky.
\newblock Cluster algebras. {I}. {F}oundations.
\newblock {\em J. Amer. Math. Soc.}, 15(2):497--529 (electronic), 2002.

\bibitem[Gai03]{Gaiffi}
Giovanni Gaiffi.
\newblock Models for real subspace arrangements and stratified manifolds.
\newblock {\em Int. Math. Res. Not.}, (12):627--656, 2003.

\bibitem[Gal21]{crit_tnn}
Pavel Galashin.
\newblock {Totally nonnegative critical varieties}.
\newblock {\em Int. Math. Res. Not. IMRN, to appear. \arxiv{2110.08548v1}},
  2021.

\bibitem[Gal23]{crit}
Pavel Galashin.
\newblock Critical varieties in the {G}rassmannian.
\newblock {\em Comm. Math. Phys.}, 401(3):3277--3333, 2023.

\bibitem[GKL19]{GKL3}
Pavel Galashin, Steven~N. Karp, and Thomas Lam.
\newblock {Regularity theorem for totally nonnegative flag varieties}.
\newblock {\em J. Amer. Math. Soc., to appear. \arxiv{1904.00527v3}}, 2019.

\bibitem[GP19]{R_systems}
Pavel Galashin and Pavlo Pylyavskyy.
\newblock {$R$}-systems.
\newblock {\em Selecta Math. (N.S.)}, 25(2):Paper No. 22, 2019.

\bibitem[Hai84]{Haiman}
Mark Haiman.
\newblock {Constructing the associahedron}.
\newblock Preprint, 1984.

\bibitem[Har77]{Hartshorne}
Robin Hartshorne.
\newblock {\em Algebraic geometry}.
\newblock Graduate Texts in Mathematics, No. 52. Springer-Verlag, New
  York-Heidelberg, 1977.

\bibitem[Het20]{Hetyei}
Gábor Hetyei.
\newblock {The type $B$ permutohedron and the poset of intervals as a
  Tchebyshev transform}.
\newblock {\em \arxiv{2007.07362v2}}, 2020.

\bibitem[JS14]{JoSa}
Katharina Jochemko and Raman Sanyal.
\newblock Arithmetic of marked order polytopes, monotone triangle reciprocity,
  and partial colorings.
\newblock {\em SIAM J. Discrete Math.}, 28(3):1540--1558, 2014.

\bibitem[Kon99]{Kontsevich}
Maxim Kontsevich.
\newblock Operads and motives in deformation quantization.
\newblock {\em Lett. Math. Phys.}, 48(1):35--72, 1999.

\bibitem[Lee89]{Lee}
Carl~W. Lee.
\newblock The associahedron and triangulations of the {$n$}-gon.
\newblock {\em European J. Combin.}, 10(6):551--560, 1989.

\bibitem[LP07]{LaPo_alcoved}
Thomas Lam and Alexander Postnikov.
\newblock Alcoved polytopes. {I}.
\newblock {\em Discrete Comput. Geom.}, 38(3):453--478, 2007.

\bibitem[LP16a]{LP_LP}
Thomas Lam and Pavlo Pylyavskyy.
\newblock Laurent phenomenon algebras.
\newblock {\em Camb. J. Math.}, 4(1):121--162, 2016.

\bibitem[LP16b]{LP_linear}
Thomas Lam and Pavlo Pylyavskyy.
\newblock Linear {L}aurent phenomenon algebras.
\newblock {\em Int. Math. Res. Not. IMRN}, (10):3163--3203, 2016.

\bibitem[LTV10]{LTV}
Pascal Lambrechts, Victor Turchin, and Ismar Voli\'{c}.
\newblock Associahedron, cyclohedron and permutohedron as compactifications of
  configuration spaces.
\newblock {\em Bull. Belg. Math. Soc. Simon Stevin}, 17(2):303--332, 2010.

\bibitem[Oda88]{Oda}
Tadao Oda.
\newblock {\em Convex bodies and algebraic geometry}, volume~15 of {\em
  Ergebnisse der Mathematik und ihrer Grenzgebiete (3) [Results in Mathematics
  and Related Areas (3)]}.
\newblock Springer-Verlag, Berlin, 1988.
\newblock An introduction to the theory of toric varieties, Translated from the
  Japanese.

\bibitem[Per02]{Per1}
Grisha Perelman.
\newblock The entropy formula for the {R}icci flow and its geometric
  applications.
\newblock {\em \arxiv{math/0211159}}, 2002.

\bibitem[Per03a]{Per2}
Grisha Perelman.
\newblock Finite extinction time for the solutions to the {R}icci flow on
  certain three-manifolds.
\newblock {\em \arxiv{math/0307245}}, 2003.

\bibitem[Per03b]{Per3}
Grisha Perelman.
\newblock {R}icci flow with surgery on three-manifolds.
\newblock {\em \arxiv{math/0303109}}, 2003.

\bibitem[Pos05]{Pos_affine}
Alexander Postnikov.
\newblock Affine approach to quantum {S}chubert calculus.
\newblock {\em Duke Math. J.}, 128(3):473--509, 2005.

\bibitem[Pos09]{Pos_perm}
Alexander Postnikov.
\newblock Permutohedra, associahedra, and beyond.
\newblock {\em Int. Math. Res. Not. IMRN}, (6):1026--1106, 2009.

\bibitem[PPPP19]{PPPP}
Arnau Padrol, Yann Palu, Vincent Pilaud, and Pierre-Guy Plamondon.
\newblock {Associahedra for finite type cluster algebras and minimal relations
  between $\mathbf{g}$-vectors}.
\newblock {\em \arxiv{1906.06861v3}}, 2019.

\bibitem[PRW08]{PRW}
Alex Postnikov, Victor Reiner, and Lauren Williams.
\newblock Faces of generalized permutohedra.
\newblock {\em Doc. Math.}, 13:207--273, 2008.

\bibitem[Sac23]{Sack}
Andrew Sack.
\newblock {A realization of poset associahedra}.
\newblock {\em \arxiv{2301.11449v2}}, 2023.

\bibitem[Sim03]{Simion}
Rodica Simion.
\newblock A type-{B} associahedron.
\newblock volume~30, pages 2--25. 2003.
\newblock Formal power series and algebraic combinatorics (Scottsdale, AZ,
  2001).

\bibitem[Sin04]{Sinha}
Dev~P. Sinha.
\newblock Manifold-theoretic compactifications of configuration spaces.
\newblock {\em Selecta Math. (N.S.)}, 10(3):391--428, 2004.

\bibitem[Sma61]{Sma}
Stephen Smale.
\newblock Generalized {P}oincar\'e's conjecture in dimensions greater than
  four.
\newblock {\em Ann. of Math. (2)}, 74:391--406, 1961.

\bibitem[Sta63]{Stasheff}
James~Dillon Stasheff.
\newblock Homotopy associativity of {$H$}-spaces. {I}, {II}.
\newblock {\em Trans. Amer. Math. Soc. 108 (1963), 275-292; ibid.},
  108:293--312, 1963.

\bibitem[Sta86]{Stanley_two}
Richard~P. Stanley.
\newblock Two poset polytopes.
\newblock {\em Discrete Comput. Geom.}, 1(1):9--23, 1986.

\bibitem[SW12]{StWi}
Jessica Striker and Nathan Williams.
\newblock Promotion and rowmotion.
\newblock {\em European J. Combin.}, 33(8):1919--1942, 2012.

\bibitem[Tam51]{Tamari}
Dov Tamari.
\newblock {\em Mono\"{\i}des pr\'{e}ordonn\'{e}s et cha\^{\i}nes de {M}alcev}.
\newblock Th\`ese, Universit\'{e} de Paris, 1951.

\bibitem[Wac07]{Wachs}
Michelle~L. Wachs.
\newblock Poset topology: tools and applications.
\newblock In {\em Geometric combinatorics}, volume~13 of {\em IAS/Park City
  Math. Ser.}, pages 497--615. Amer. Math. Soc., Providence, RI, 2007.

\bibitem[Zie95]{Ziegler}
G\"{u}nter~M. Ziegler.
\newblock {\em Lectures on polytopes}, volume 152 of {\em Graduate Texts in
  Mathematics}.
\newblock Springer-Verlag, New York, 1995.

\end{thebibliography}

\end{document}